\documentclass[11pt, leqno]{amsart}
\usepackage{amsthm,amsmath,amssymb,amscd,verbatim}
\usepackage[curve,matrix,arrow,frame,tips]{xy}
\usepackage{graphicx}
\usepackage{hyperref}

\theoremstyle{plain}
\newtheorem{theorem}{Theorem}[section]
\newtheorem{lemma}[theorem]{Lemma}
\newtheorem{cor}[theorem]{Corollary}
\newtheorem{prop}[theorem]{Proposition}
\newtheorem{conj}[theorem]{Conjecture}

\theoremstyle{remark}

\theoremstyle{definition}

\newtheorem{defn}[theorem]{Definition}
\newtheorem{Remark}[theorem]{Remark}

\numberwithin{equation}{section}

\newcommand\hfld[2]{\smash{\mathop{\hbox to 10mm{\rightarrowfill}}
     \limits^{\scriptstyle#1}_{\scriptstyle#2}}}
\newcommand\hflg[2]{\smash{\mathop{\hbox to 10mm{\leftarrowfill}}
     \limits^{\scriptstyle#1}_{\scriptstyle#2}}}

\author{Thomas J. Haines}

\begin{document}

\thanks{This research has been partially supported by NSF grant DMS-0901723.}

\date{}

\title{On Satake parameters for representations with parahoric fixed vectors}

\maketitle

\begin{abstract}
This article, a continuation of \cite{HRo}, constructs the Satake parameter for any irreducible smooth $J$-spherical representation of a $p$-adic group, where $J$ is any parahoric subgroup. This parametrizes such representations when $J$ is a special maximal parahoric subgroup. The main novelty is for groups which are not quasi-split, and the construction should play a role in formulating a geometric Satake isomorphism for such groups over local function fields.
\end{abstract}

\section{Introduction}

Let $F$ be a nonarchimedean local field and let $W_F$ denote its Weil group, with $I_F$ its inertia subgroup and $\Phi \in W_F$ a choice of a geometric Frobenius element. Let $G$ be a connected reductive group over $F$, with complex dual group $\widehat{G}$. Let $J \subset G(F)$ be a parahoric subgroup, and let $\Pi(G/F, J)$ denote the set of isomorphism classes of smooth irreducible representations $\pi$ of $G(F)$ such that $\pi^J \neq 0$.  To $\pi \in \Pi(G/F, J)$ we will associate a {\em Satake parameter} $s(\pi)$ belonging to the variety $[\widehat{G}^{I_F} \rtimes \Phi]_{\rm ss}/ \widehat{G}^{I_F}$, where the quotient is formed using the conjugation action of $\widehat{G}^{I_F}$ on the set of semisimple elements in the coset $\widehat{G}^{I_F} \rtimes \Phi$. More precisely, we will prove the following theorem.

\begin{theorem} \label{main_thm}
There is an explicit closed subvariety $S(G) \subseteq [\widehat{G}^{I_F} \rtimes \Phi]_{\rm ss}/\widehat{G}^{I_F}$ and a canonical map $s: \Pi(G/F, J) \rightarrow S(G)$ with the following properties:
\begin{enumerate}
\item[(A)] If $J = K$ is a special maximal parahoric subgroup, the map $\pi \mapsto s(\pi)$ gives a {\em parametrization}
$$
\Pi(G/F, K) ~ \widetilde{\rightarrow} ~ S(G).
$$
\item[(B)] $S(G) =  [\widehat{G}^{I_F} \rtimes \Phi]_{\rm ss}/\widehat{G}^{I_F}$ if and only if $G/F$ is quasi-split.
\item[(C)] The parameter $s(\pi)$ predicts part of the local Langlands parameter $\varphi_\pi$ that is conjecturally attached to $\pi$: $\varphi_\pi(\Phi) = s(\pi)$ in $[\widehat{G} \rtimes \Phi]_{\rm ss}/\widehat{G}$ \textup{(}Conjecture \ref{LLC_relation}\textup{)}.
\end{enumerate}
\end{theorem}

The evidence for (C) is contained in the following result, which we prove in $\S\ref{rel_LLC_sec}$, under the assumption that inner forms of ${\rm GL}_n$ satisfy the enhancement LLC+ of the local Langlands correspondence (see \cite[$\S5.2$]{stable}).

\begin{theorem} \label{2nd_main_thm}
Conjecture \ref{LLC_relation} holds if $G$ is any inner form of ${\rm GL}_n$.
\end{theorem}

The map $\pi \mapsto s(\pi)$ is constructed as follows: to $\pi$ we associate its supercuspidal support, which by \cite[$\S11.5$]{stable} is a cuspidal pair $(M, \chi)_G$ with $M = {\rm Cent}_G(A)$ a minimal $F$-Levi subgroup of $G$ and $\chi \in X^{\rm w}(M) = {\rm Hom}_{\rm grp}(M(F)/M(F)_1, \mathbb C^\times)$ a weakly unramified character on $M(F)$ (in the terminology of \cite[3.3.1]{stable}). Here $M(F)_1$ is the kernel of the Kottwitz homomorphism \cite[$\S7$]{Ko97}, the theory of which gives an isomorphism
\begin{equation} \label{Kottwitz_isom}
\kappa_M: M(F)/M(F)_1 ~ \widetilde{\rightarrow} ~ X^*(Z(\widehat{M})^{I_F}_{\Phi}).
\end{equation}
Recalling that $\chi$ is determined by $\pi$ up to conjugation by the relative Weyl group $W(G, A)$, we can view the supercuspidal support of $\pi$ as an element in the complex affine variety $(Z(\widehat{M})^{I_F})_\Phi/W(G,A)$. Thus $\pi \mapsto \chi$ gives a map
\begin{equation} \label{map_1}
\Pi(G/F, J) \rightarrow (Z(\widehat{M})^{I_F})_{\Phi}/W(G, A).
\end{equation}
On the other hand, if $(G, \Psi)$ is an $F$-inner form of a quasi-split group $G^*$, and if $A^* \subset T^* \subset G^*$ are data parallel to $A \subset M \subset G$, then the theory of the normalized transfer homomorphisms $\tilde{t}_{A^*, A}$ from $\S\ref{trans_sec}$ together with the material in $\S\ref{param_space_sec5}, \ref{param_space_sec6}$ gives rise to a canonical closed immersion
\begin{equation} \label{map_2}
\xymatrixcolsep{4pc}\xymatrix{
(Z(\widehat{M})^{I_F})_{\Phi}/W(G,A) \ar@{^{(}->}[r]^{(\ref{key_defn})}  & [\widehat{G}^{I_F} \rtimes \Phi]_{\rm ss}/\widehat{G}^{I_F}.}
\end{equation}
As explained in $\S\ref{param_const_gen}$, the composition $s$ of (\ref{map_1}) with (\ref{map_2}) is completely canonical (independent of the choice of $A$), and $S(G)$ is defined to be its image. 

The following result gives two more conceptual descriptions of $S(G)$. Let $\mathcal N(\widehat{G}^{I_F})$ denote the set of nilpotent elements in ${\rm Lie}(\widehat{G}^{I_F})$. We call $(\hat{g} \rtimes \Phi, x) \in [\widehat{G}^{I_F} \rtimes \Phi]_{\rm ss} \times \mathcal N(\widehat{G}^{I_F})$ a {\em regular} $\Phi$-{\em admissible pair} if 
\begin{itemize}
\item ${\rm Ad}(\hat{g} \rtimes \Phi)(x) = q_F \, x$, where $q_F$ is the cardinality of the residue field of $F$;
\item $\Phi(x) = x$;
\item $x$ is a {\em principal} nilpotent element in ${\rm Lie}(\widehat{G}^{I_F})$.
\end{itemize}
Denote the set of such pairs by ${\mathcal P}_{\rm reg}^\Phi(\widehat{G}^{I_F})$. 

\begin{theorem} \label{explicit_S(G)} Let $M = {\rm Cent}_G(A)$ be any minimal $F$-Levi subgroup of $G$, and let $\widehat{M} \subset \widehat{G}$ be a corresponding $W_F$-stable Levi subgroup of the complex dual group $\widehat{G}$.  The following are equivalent for an element $\hat{g} \rtimes \Phi \in [\widehat{G}^{I_F} \rtimes \Phi]_{\rm ss}/\widehat{G}^{I_F}$:
\begin{enumerate} 
\item[(i)] $\hat{g} \rtimes \Phi \in S(G)$;
\item[(ii)] $\hat{g} \rtimes \Phi$ is $\widehat{G}^{I_F}$-conjugate to the first coordinate of a pair in ${\mathcal P}_{\rm reg}^\Phi(\widehat{M}^{I_F})$;
\item[(iii)] $\hat{g} \rtimes \Phi$ is $\widehat{G}^{I_F}$-conjugate to the image under the natural map
$$
[\widehat{M^*}^{I_F} \rtimes \Phi]_{\rm ss}/\widehat{M^*}^{I_F} = [\widehat{M}^{I_F} \rtimes \Phi]_{\rm ss}/\widehat{M}^{I_F} \rightarrow [\widehat{G}^{I_F} \rtimes \Phi]_{\rm ss}/\widehat{G}^{I_F}$$ of the Satake parameter $s(\sigma^*)$ of some weakly unramified twist $\sigma^*$ of the Steinberg representation for $M^*$. Here $M^*$ is a quasi-split $F$-inner form of $M$ \textup{(}see Remark \ref{dual_innertwist_rem}\textup{)}.
\end{enumerate}
\end{theorem}
In formulating (ii) we used implicitly that $\widehat{G}^{I_F}$ and $\widehat{M}^{I_F}$ are reductive groups and that $\widehat{G}^{I_F} = Z(\widehat{G})^{I_F} \, \widehat{G}^{I_F, \circ}$; see $\S\S \ref{fixed-pt_sec}, \ref{param_space_sec5}$.

Let us consider some special cases and history. The most important case is where $J = K$ is a special maximal parahoric subgroup. If $G/F$ is unramified, then such a $K$ is automatically a special {\em maximal compact} subgroup (cf.~\cite{HRo}), and  $S(G) = [\widehat{G} \rtimes \Phi]_{\rm ss}/\widehat{G}$, and the parametrization in (A) is classical (cf.~\cite{Bor}). If $G/F$ is only quasi-split and tamely ramified (i.e.~split over a tamely ramified extension of $F$), then the parametrization in (A) was proved by M.~Mishra \cite{Mis} and some similar results were also obtained by X.~Zhu \cite{Zhu}.

The same ideas show how to construct the $s$-parameter in the hypothetical Deligne-Langlands triple $(s,u, \rho)$ one could hope to associate to parahoric-spherical representations of general connected reductive groups; see $\S\ref{rel_LLC_sec}$, where item (C) is also explained. It should be stressed that throughout this article,  ``parahoric'' should be understood in the sense of Bruhat-Tits \cite{BT2}, as the $\mathcal O_F$-points of a {\em connected} group scheme over $\mathcal O_F$. So for example an Iwahori subgroup here is somewhat smaller than the ``naive'' notion that sometimes appears in the literature under the same name, and therefore the Iwahori-Hecke algebra and its center are slightly larger (cf. \cite{Corr} and \cite[Appendix]{stable}). 

It was clear that some kind of parametrization like that in (A) should hold, after the author and S.~Rostami proved in \cite{HRo} the general form of the Satake isomorphism
\begin{equation} \label{Sat_isom_intro}
\mathcal H(G(F), K)  \cong \mathbb C[ (Z(\widehat{M})^{I_F})_\Phi/W(G,A)].
\end{equation}
It was also clear at that time that this isomorphism is the right one to ``categorify'', in other words it should be the function-theoretic shadow of a geometric Satake isomorphism \`{a} la \cite{MV} for $G/F$ and $K$, once such an isomorphism is properly formulated (of course here we assume $F = \mathbb F_q(\!(t)\!)$). In the meantime, progress in exactly this direction has been made: X.~Zhu \cite{Zhu} proved a geometric Satake isomorphism extending (\ref{Sat_isom_intro}) for quasi-split and tamely ramified $G$ (and {\em very special $K$}, in Zhu's terminology). This was recently generalized by T.~Richarz \cite{Ri}, who effectively removed the ``tamely ramified'' hypothesis from Zhu's result, while still assuming $G$ is quasi-split and $K$ is very special.

One obstacle to formulating a geometric Satake isomorphism when $G/F$ is not quasi-split is the lack of a suitable link between the right hand side of (\ref{Sat_isom_intro}) and the $L$-group $^LG := \widehat{G} \rtimes W_F$. We are proposing that (\ref{map_2}) provides the sought-after link. In order to fully justify this idea, it would be important to establish a suitable ``categorification''  of the normalized transfer homomorphisms, of the subvariety $S(G)$, and of the closed immersion (\ref{map_2}). The author hopes to return to these matters in future work.

\medskip

Here is an outline of the contents of this article. In $\S\ref{notation_sec}$ we recall some notation that is used throughout the paper. In $\S3$ we recall the parametrization of $\Pi(G/F, K)$ that is a consequence of (\ref{Sat_isom_intro}) and other results from \cite{HRo}. The purpose of $\S\ref{fixed-pt_sec}$ is to lay some groundwork needed in order to prove properties of $\widehat{G}^{I_F}$ (e.g.~it is reductive; analysis of its group of connected components) which are needed in $\S\S \ref{param_space_sec5}, \ref{param_space_sec6},  \ref{2nd_param_sec}$ on the parameter space $[\widehat{G}^{I_F} \rtimes \Phi]_{\rm ss}/\widehat{G}^{I_F}$. Those sections handle the construction of $\pi \mapsto s(\pi)$ when $G/F$ is quasi-split. Section $\S\ref{trans_sec}$ provides the key ingredients (transfer homomorphisms, etc.) needed to extend the construction to the general case, which is done in $\S\S \ref{param_const_gen}, \ref{2nd_param_gen_sec}$. Theorem \ref{main_thm} parts (A) and (B) are proved in $\S\ref{2nd_param_gen_sec}$. We prove Theorem \ref{explicit_S(G)} in $\S \ref{pf_exp_S(G)_sec}$, relying on the key Lemma \ref{nilpotent_lem} proved at the end of $\S$\ref{fixed-pt_sec}. Finally, in $\S\S \ref{transfer_map_sec}, \ref{rel_LLC_sec}$ we explain the connection of the Satake parameters to the (conjectural) local Langlands and Jacquet-Langlands correspondences, and also justify (C) by proving Theorem \ref{2nd_main_thm}.

\section{Notation and conventions} \label{notation_sec}

We denote the absolute Galois group of $F$ by $\Gamma := {\rm Gal}(F^s/F)$, where $F^s$ is some separable closure of $F$, fixed once and for all.

If $G$ is any connected reductive group over a nonarchimedean field $F$, and if $J \subset G(F)$ is any compact open subgroup, then $\mathcal H(G(F), J) := C_c(J\backslash G(F)/J)$, a $\mathbb C$-algebra when endowed with the convolution $*$ defined by using the Haar measure on $G(F)$ which gives $J$ volume 1.  We write $\mathcal Z(G(F), J)$ for the center of $\mathcal H(G(F), J)$.

For any $F$-Levi subgroup $M$ and $F$-parabolic subgroup $P$ with unipotent radical $N$ and Levi decomposition $P = MN$, we define for $m \in M(F)$ the usual modulus function 
$$\delta_P(m) := |{\rm det}({\rm Ad}(m); {\rm Lie}\, N(F))|_F,$$ where $| \cdot |_F$ is the normalized absolute value on $F$. Then for any admissible representation $\sigma$ of $M(F)$, we set $i^G_P(\sigma) := {\rm Ind}_{P(F)}^{G(F)}(\sigma \otimes \delta_P^{1/2})$ where ${\rm Ind}^?_?(?)$ denotes usual (unnormalized) induction.

We use $^xY$ to denote $x Y x^{-1}$ for $x$ an element and $Y$ a subset of some group. If $f$ is a function on that group, $^xf$ will be the function $y \mapsto f(x^{-1}yx)$. 
 
We will use Kottwitz' conventions on dual groups $\widehat{G}$ and their $\Gamma$-actions, see  \cite[$\S1$]{cusptemp}. 

\section{First parametrization of $K$-spherical representations} \label{1st_param_sec}

Fix a special maximal parahoric subgroup $K \subset G(F)$. In $G$, choose any maximal $F$-split torus $A$ whose associated apartment in the Bruhat-Tits building $\mathcal B(G_{\rm ad}, F)$ contains the special vertex associated to $K$.  Let $M := {\rm Cent}_G(A)$  be the centralizer of $A$, a minimal $F$-Levi subgroup. Following \cite{stable}, we call the group of homomorphisms $M(F)/M(F)_1 \rightarrow \mathbb C^\times$ the group $X^{\rm w}(M)$ of {\em weakly unramified} characters on $M(F)$. The Kottwitz homomorphism \cite[$\S7$]{Ko97} induces an isomorphism $M(F)/M(F)_1 \cong X^*(Z(\widehat{M})^{I_F}_\Phi)$, so that $X^{\rm w}(M) \cong (Z(\widehat{M})^{I_F})_\Phi$, a diagonalizable group over $\mathbb C$.

Given $\chi \in X^{\rm w}(M)$, the Iwasawa decomposition of \cite[Cor.~9.1.2]{HRo} allows us to define an element $\Phi_{K, \chi} \in i^G_P(\chi)^K$ by
$$
\Phi_{K, \chi}(mnk) = \delta_P^{1/2}(m) \, \chi(m)
$$
for $m \in M(F)$, $n \in N(F)$, and $k \in K$ (here $N$ is the unipotent radical of an $F$-parabolic subgroup $P$ having $M$ as Levi factor). Then define the spherical function 
$$
\Gamma_\chi(g) = \int_{K} \Phi_{K, \chi}(kg) \, dk
$$
where ${\rm vol}_{dk}(K) = 1$. Let $\pi_{\chi}$ denote the smallest  $G$-stable subspace of the right  regular representation of $G(F)$ on $C^{\infty}(G(F))$ containing $\Gamma_\chi$. Then, as in \cite[$\S4.4$]{Car}, we see that $\pi_\chi$ is irreducible, that $\pi_{\chi} \cong \pi_{\chi'}$ iff $\chi = \, ^w\chi'$ for some $w \in W(G,A)$, and that every element of $\Pi(G/F,K)$ is isomorphic to some $\pi_\chi$.  Thus we have the following first parametrization of $\Pi(G/F,K)$.

\begin{prop}
The map $\chi \mapsto \pi_\chi$ sets up a 1-1 correspondence
\begin{equation} \label{1st_param}
(Z(\widehat{M})^{I_F})_\Phi/ W(G,A) ~ \widetilde{\longrightarrow} ~ \Pi(G/F,K).
\end{equation}
Moreover, if $f \in \mathcal H(G(F), K)$, then $\pi_\chi(f)$ acts on $\pi_\chi^K$ by the scalar $S(f)(\chi)$, where $S$ is the Satake isomorphism 
\begin{equation} \label{Sat_isom}
S  \, : \, \mathcal H(G(F), K) ~ \widetilde{\longrightarrow} ~ \mathbb C[(Z(\widehat{M})^{I_F})_\Phi/W(G,A)]
\end{equation}
of \cite[Thm. 1.0.1]{HRo}. 
Here the right hand side denotes the ring of regular functions on the affine variety $(Z(\widehat{M})^{I_F})_\Phi/W(G,A)$.
\end{prop}

\section{Fixed-point subgroups under finite groups of automorphisms} \label{fixed-pt_sec}

Steinberg \cite{St} proved fundamental results on cyclic groups of automorphisms of a simply connected semisimple algebraic group. In the same context, when the generator of the cyclic group comes from a diagram automorphism, Springer \cite{Sp2} supplemented Steinberg's results by, among other things, giving information about the root data of the fixed-point group. 
The aim here is to extend some of the results of Steinberg and Springer to finite groups of automorphisms of a reductive group. The following might be known, but we include a complete proof here due to the lack of a suitable reference.

Notation: If a group $J$ acts by automorphisms on an algebraic group $P$, we write $P^\circ$ for the neutral component of $P$ and often write $P^{J, \circ}$ instead of $(P^J)^\circ$.

\begin{prop} \label{Steinberg_style_prop} Let $H$ be a possibly disconnected reductive group over an algebraically closed field  $k$. Assume that a finite group $I$ acts by automorphisms on $H$ and preserves a splitting $(T, B, X)$, consisting of a Borel subgroup $B$, a maximal torus $T$ in $B$, and a principal nilpotent element $X = \sum_{\alpha \in \Delta(T,B)} X_\alpha$ for some non-zero elements $X_\alpha \in ({\rm Lie} \, H^\circ)_\alpha$ indexed by the $B$-positive simple roots $\Delta(T,B)$ in $X^*(T)$. Let $U$ be the unipotent radical of $B$, and let $N = N(H^\circ, T)$ be the normalizer of $T$ in $H^\circ$. Then:
\begin{enumerate}
\item[(a)] The algebraic group $H^I$ is reductive with identity component $(H^I)^\circ = [(H^\circ)^I]^\circ$ and with splitting $(T^{I,\circ}, B^{I,\circ}, X^I)$, where $B^{I,\circ} = T^{I,\circ} U^I$ and where $X^I$ is a principal nilpotent in ${\rm Lie}(H^{I,\circ})$ constructed from $X$. (If ${\rm char}(k) \neq 2$, then $X^I = X$.)
\item[(b)] We have $T^I \cap H^{I,\circ} = T^{I,\circ}$ and $N^I \cap H^{I,\circ} = N( H^{I,\circ}, T^{I,\circ})$.
\item[(c)] If $W := W(H^\circ, T) := N/T$,  then every element of $W^I$ has a representative in $N^I \cap H^{I,\circ}$, and thus $W( H^{I,\circ}, T^{I,\circ}) = W^I$.
\item[(d)] The inclusion $T \hookrightarrow H^\circ$ induces a bijection $\pi_0\,T^I ~ \widetilde{\rightarrow} ~ \pi_0 \, (H^\circ)^I$.
\end{enumerate} 
 \end{prop}

Before beginning the proof, note that giving the data of $X$ is equivalent to giving the data $\{ x_\alpha \}_{\alpha \in \Delta(T,B)}$ of root group homomorphisms $x_\alpha: {\mathbb G}_a ~ \widetilde{\rightarrow} ~ U_\alpha$, where $U_\alpha$ is the maximal connected unipotent subgroup of $H^\circ$ normalized by $T$ and with ${\rm Lie}(U_\alpha) = ({\rm Lie}\, H^\circ)_\alpha$. This is because the Lie functor gives an isomorphism ${\rm Isom}_{k-{\rm Grp}}({\mathbb G}_a, U_\alpha) = {\rm Isom}_k(k, 
{\rm Lie}(U_\alpha))$.

\begin{proof}
Quite generally $(H^I)^\circ$ contains $[(H^\circ)^I]^\circ$ with finite index, and as both are connected algebraic groups, they coincide.  

\begin{lemma} \label{root_lem}
Let $\Psi$ be a reduced root system in a real vector space $V$, with set of simple roots $\Delta$.  Suppose $I$ is a finite group of automorphisms of $V$ which preserves $\Psi$ and $\Delta$.  Let $\bar{\alpha} \in V$ denote the average of the $I$-orbit of $\alpha \in \Psi$, and let $\Psi^I = \{ \bar{\alpha} ~ | ~ \alpha \in \Psi\}$ and $\Delta^I = \{ \bar{\alpha} ~ | ~ \alpha \in \Delta \}$. Then 
\begin{enumerate}
\item[(1)] $\Psi^I$ is a possibly non-reduced root system in $V^I$ with set of simple roots $\Delta^I$;
\item[(2)] $W(\Psi^I) = W(\Psi)^I$, where $W(\Sigma)$ denotes the Weyl group of a root system $\Sigma$.
\end{enumerate}
\end{lemma}

\begin{proof}
First assume $I = \langle \tau \rangle$. Let $\Psi_\tau \subseteq \Psi^\tau$ be defined by discarding those elements of $\Psi^\tau$ which are smaller multiples of others. Then \cite[1.32, 1.33]{St} shows that $\Psi_\tau$ is a root system with Weyl group $W(\Psi)^\tau$.  The only difference between $\Psi_\tau$ and $\Psi^\tau$ is that the latter could contain $\frac{1}{2}\alpha'$ for $\alpha' \in \Psi_\tau$, and then only for a component of $\Psi$ of type $A_{2n}$. Consideration of the root system for a quasi-split unitary group in $2n+1$ variables attached to a separable quadratic extension of a $p$-adic field (cf. \cite[$\S1.15$]{Tits}) shows that adding such half-roots to a root system of form $\Psi_\tau$, still gives a root system; so $\Psi^\tau$ is indeed a root system, with simple roots $\Delta^\tau$ and with the same Weyl group $W(\Psi)^\tau$.  Note that if $|\tau|$ is odd, then $\Psi^\tau$ is again reduced (comp.~\cite[Lemma 9.2]{HN}).

Now decompose $(V, \Psi)$ into a sum of simple systems $(V_j, \Psi_j)$. The action of $I$ permutes these simple systems while the stabilizer of each component continues to act through its automorphism group. Therefore we may assume $(V,\Psi)$ is simple. Using the classification, we may assume $I$ acts through a faithful action of $\mathbb Z/2\mathbb Z, \, \mathbb Z/3\mathbb Z,$ or $S_3$ on $\Delta$. In the last case, $\Psi = D_4$ and the $I$-orbits on $\Psi^+$ and on $\Delta$ coincide with those of the subgroup $\mathbb Z/3\mathbb Z \subset S_3$. Thus we may assume $I$ is cyclic, and we may apply the preceding paragraph.
\end{proof}

We will apply Lemma \ref{root_lem} when $\Psi=\Psi(H^\circ, T)$ (resp. $\Delta = \Delta(T, B)$), the set of roots (resp.~$B$-simple roots) for $T$ in ${\rm Lie}(H^\circ)$ in the vector space $V = X^*(T) \otimes \mathbb R$. The set of positive roots $\Psi^+$ is a disjoint union of subsets $S_\alpha$, where $S_\alpha$ consists of all those positive roots whose projection to $V^I$ is proportional to that of $\alpha$ (comp.~\cite[Thm.~8.2(2')]{St}).  We will always index $S_\alpha$ by a {\em minimal} element $\alpha$ in this set (use the usual partial order on positive roots).  For example (cf.~\cite[Thm.~8.2(2')]{St}), if $I = \langle \tau \rangle$, then $S_\alpha$ comes in two types:

 \begin{enumerate}
\item[(Type 1)] $S_\alpha$ is a $\tau$-orbit $\{ \alpha, \tau \alpha, \dots \},$ no two of which add up to a root;
\item[(Type 2)] $S_\alpha = \{\alpha, \tau \alpha, \beta \}$, where $\beta := \alpha + \tau \alpha$ is a root; this occurs only in type $A_{2n}$.
\end{enumerate}

For a root system of the form $\Psi^\tau$, we write $(\Psi^\tau)^{\rm red}$ (resp.~$(\Psi^\tau)_{\rm red})$) for the root system we get by discarding vectors from $\Psi^\tau$ which are shorter (resp.~longer) multiples of others.  For example, $\Psi_\tau = (\Psi^\tau)^{\rm red}$.

We now make the following temporary assumptions:
\begin{itemize}
\item[(i)] $H^{I,\circ}$ is reductive with splitting $( T^{I,\circ}, B^{I,\circ}, X^I)$, where $X^I$ denotes a principal nilpotent element of ${\rm Lie}(H^{I,\circ})$ constructed from $X$. (If ${\rm char}(k) \neq 2$, then $X^I = X$.)
\item[(ii)] $\Psi(H^{I,\circ}, T^{I,\circ})$ can be identified with $(\Psi^I)_{\rm red}$ if ${\rm char}(k) \neq 2$ and with $(\Psi^I)^{\rm red}$ if ${\rm char}(k) = 2$.  
\item[(iii)] If ${\rm char}(k) \neq 2$, the group $U^I$ is the product of certain subgroups $U_{\bar{\alpha}}$ indexed by the various subsets $S_\alpha$, where ${\rm Ad}(T^{I,\circ})$ acts via $\bar{\alpha}$ on ${\rm Lie}(U_{\bar{\alpha}}) \subset {\rm Lie}(H^{I,\circ})$. If $\alpha \in \Delta(T, B)$, then $U_{\bar{\alpha}} \cong {\mathbb G}_a$, and in general $U_{\bar{\alpha}}$ is either trivial or is isomorphic to ${\mathbb G}_a$. In particular $U^I$ is connected. Furthermore, $U_{\bar{\alpha}}$ is contained in the product of the groups $U_{\alpha'}$ for $\alpha' \in S_\alpha$. If ${\rm char}(k) = 2$, the same statements hold with $U_{\bar{\alpha}}$ replaced by $U_{2 \bar{\alpha}}$ when $S_\alpha$ is of Type 2.
\end{itemize}

\begin{Remark}
Property (iii) automatically implies another property:
\begin{itemize}
\item[(iv)] The map $N^I \rightarrow W(\Psi)^I, \, n \mapsto w_n,$ has the property that for every subset $S_\alpha$, we have $n U_{i\bar{\alpha}} n^{-1} = U_{w_n(i\bar{\alpha})}$, for $i \in \{1,2\}$.
\end{itemize}
\end{Remark}

\begin{Remark} \label{nontriv_rem}
Later we will see that $U_{\bar{\alpha}}$ (resp. $U_{2\bar{\alpha}}$) is isomorphic to ${\mathbb G}_a$ for {\em all} $\alpha \in \Psi^+$ representing a set $S_\alpha$, not just for $\alpha \in \Delta(T,B)$.
\end{Remark}

\noindent {\bf Lifting step.} Let $s_{\bar{\alpha}} \in W(\Psi)^I$ be the reflection corresponding to a {\em simple} root $\bar{\alpha}$ for some $\alpha \in \Delta(T, B)$ (cf.~Lemma \ref{root_lem}). We wish to show it can be lifted to an element in $N^I \cap H^{I,\circ}$. Using (iii), we may choose $u \in U_{\bar{\alpha}} \backslash \{1\}$ if ${\rm char}(k) \neq 2$ or $S_\alpha$ is Type 1 (resp.~$u \in U_{2\bar{\alpha}} \backslash \{1\}$ if ${\rm char}(k) = 2$ and $S_\alpha$ is Type 2). Using the Bruhat decomposition for $\overline{U}$ in place of $U$, we can write uniquely $u = u_1 n u_2$, where $u_2 \in \overline{U}$, $n \in N$, and $u_1 \in \overline{U} \cap n U n^{-1}$; since $u$ is $I$-fixed, $u_1, n, u_2$ are too.  The element $n$ belongs to $N^I \cap \overline{U}^I U_{i\bar{\alpha}} \overline{U}^I$ ($i \in \{1,2\}$) and thus to $N^I \cap H^{I,\circ}$ by (iii). The element $w_n \in W \cong W(\Psi)$ to which $n$ projects is different from $1$, is fixed by $I$, and is in the group generated by the reflections $s_{\tau(\alpha)}, \, \tau \in I$. For the last statement, use (iii) and \cite[9.2.1]{Sp1} to show that $n \in \langle U_{\pm \alpha'} \rangle_{\alpha'} \cdot \overline{U}$ where $\alpha'$ ranges over elements in $S_\alpha$ (a Levi subset of $\Psi^+$ when $\alpha \in \Delta$), and then use the Bruhat decomposition again. Let $\Psi^{I,+}$ denote the positive roots of $\Psi^I$. If $w_n \neq s_{\bar{\alpha}}$ then $w_n$ sends some root in $\Psi^{I,+} \backslash \{ \bar{\alpha}, 2\bar{\alpha} \}$ to $-\Psi^{I,+}$. Then as an element of $W(\Psi)$, $w_n$ makes negative some positive root outside $S_\alpha$, in violation of $w_n \in \langle s_{\tau(\alpha)}, \tau \in I \rangle$. Thus $w_n = s_{\bar{\alpha}}$, and $n$ is the desired lift of $s_{\bar{\alpha}}$.

By Lemma \ref{root_lem}, $W(\Psi)^I = W(\Psi^I)$ and so any element $w \in W(\Psi)^I$ is a product of elements $s_{\bar{\alpha}}$ as above; hence $w$ can be lifted to $N^I \cap H^{I,\circ}$.  This proves part of (c). Since $N^I$ clearly maps to $W^I \cong W(\Psi)^I$, the proof also shows that $N^I \subset \langle T^I, U^I, \overline{U}^I \rangle$. As $(H^\circ)^I = \langle N^I, U^I \rangle$ by the Bruhat decomposition of $H^\circ$, we obtain 
\begin{equation} \label{H^I-generators}
(H^\circ)^I = \langle T^I, U^I, \overline{U}^I \rangle.
\end{equation}
From this we see  $[(H^\circ)^I]^\circ$ contains the connected subgroup $\langle (T^I)^\circ, U^I, \overline{U}^I \rangle$ with finite index, and so 
\begin{equation} \label{H^Icirc-generators}
[(H^\circ)^I]^\circ = \langle (T^I)^\circ, U^I, \overline{U}^I \rangle.
\end{equation} 
Hence 
\begin{equation} \label{pi_0(H^I)}
(H^\circ)^I = T^I \cdot [(H^\circ)^I]^\circ.
\end{equation}
We claim that $(T^I)^\circ = T^I \cap [(H^\circ)^I]^\circ$. The inclusion "$\subseteq$" is clear. As for the other, using (\ref{H^Icirc-generators}) it is enough to show that $T^I \cap \langle U^I , \overline{U}^I \rangle \subset (T^I)^\circ$. But $\langle U^I, \overline{U}^I \rangle$ lies in the image of  $[(H^\circ)_{\rm sc}]^I \rightarrow [(H^\circ)_{\rm der}]^I$, and so we are reduced to the case where $H^\circ$ is simply connected.\footnote{Since it fixes a splitting, the action of $I$ on $(H^\circ)_{\rm der}$ can be lifted to give a compatible action on its simply-connected cover $(H^\circ)_{\rm sc}$, for example by using the Isomorphism Theorem \cite[9.6.2]{Sp1}.} 
In that case $X^*(T)$ has a $\mathbb Z$-basis permuted by $I$, and so $T^I$ is already connected, and the result is obvious. 

It is clear that $N^I \cap H^{I,\circ} \subseteq N(H^{I,\circ}, T^{I,\circ})$.  We claim that equality holds.
The Bruhat decomposition for the reductive group $H^\circ$ implies that an element of $[(H^\circ)^I]^\circ$ decomposes uniquely in the form $u n v$ where $n \in N^I \cap [(H^\circ)^I]^\circ$, $v \in U^I$ and $u \in U^I \cap \, ^n\overline{U}^I$. (Note $n$ automatically belongs to $[(H^\circ)^I]^\circ$ since $U^I$ and $U^I \cap \, ^n\overline{U}^I$ are connected, as follows from (iii, iv).) Since by (i,iii) $[(H^\circ)^I]^\circ$ is reductive with Borel subgroup $B^{I,\circ} = T^{I,\circ} U^I$, the element also decomposes as $u_1 n_1 v_1$ where $n_1 \in N(H^{I,\circ}, T^{I,\circ})$, $v_1 \in U^I$, and $u_1 \in U^I \cap \, ^{n_1}\overline{U}^I$. Comparing these decompositions, we see $n = n_1$, i.e.,  $N(H^{I,\circ}, T^{I,\circ}) = N^I \cap H^{I,\circ}$.  

As $N^I \cap H^{I,\circ}$ surjects onto $W^I$, we deduce $N^I \cap H^{I,\circ}/T^I \cap H^{I,\circ} ~ \widetilde{\rightarrow} ~ W^I$. The above paragraphs show the left hand side is $W(H^{I,\circ}, T^{I,\circ})$.

At this point we have proved (a-d) assuming (i-iii).\footnote{In fact only (i,iii) are needed in the argument.} Now we need to prove (i-iii).  We first consider the case where $I$ is generated by a single element $\tau$. We use results of Steinberg \cite{St}, especially 8.2, 8.3. (Much of what we need also appears in \cite[$\S1.1$]{KS}.) We will adapt the proof of \cite[Theorem 8.2]{St}: it assumes $H^\circ$ is semisimple and simply connected and only assumes $\tau$ fixes $T$ and $B$, but the argument carries over when $\tau$ fixes a splitting because in \cite[Theorem 8.2]{St}~step (5) we may take $t =1$. Indeed, for each $B$-positive root $\alpha \in X^*(T)$ let $x_\alpha: {\mathbb G}_a \rightarrow H^\circ$ be the corresponding root homomorphism, and write 
\begin{equation} \label{c_alpha_defn}
\tau x_{\tau\alpha}(y) = x_\alpha(c_{\tau \alpha} y)
\end{equation}
for all $y \in k$ and some constants $c_\alpha \in k^\times$. (Here $\tau \alpha$ is defined to be $\alpha \circ \tau: T \rightarrow \mathbb G_m$.) Then the hypothesis that $\tau$ fixes $X$ is equivalent to $c_\alpha = 1, \,\, \forall \alpha \in \Delta(T,B)$.  Since $\alpha(t) = c_{\tau \alpha}$ for $\alpha \in \Delta(T,B)$ by definition of $t$, we may choose $t = 1$.

Our first task is to prove (iii).
Recall that \cite{St} shows that $U^\tau$ is connected by analyzing the conditions under which an element of the form $x_\alpha(y_\alpha) x_{\tau \alpha}(y_{\tau \alpha}) \cdots$ (indices ranging over $S_\alpha$) belongs to $U^\tau$. To ease notation, write $H_1$ (resp.~$T_1$, $B_1$) for $H^{\tau,\circ}$ (resp.~$T^{\tau,\circ}$, $B^{\tau,\circ}$).  First suppose $\alpha \in \Psi^+$ represents a set $S_\alpha$ of Type 1. Consider the average $\bar{\alpha}$ of the orbit $\{ \alpha, \tau\alpha, \dots\}$. By \cite[Thm.~8.2 (2)]{St}, there are nontrivial elements in $U^\tau$ of the form $x_\alpha(y_\alpha) x_{\tau \alpha}(y_{\tau \alpha}) \cdots$ only if $c_\alpha c_{\tau \alpha} \cdots = 1$, in which case they form a 1-parameter subgroup $U_{\bar{\alpha}}$ consisting of elements of the form
\begin{equation} \label{proto_type1_rtgrp}
x^{H_1}_{\bar{\alpha}}(y) = x_\alpha(y) x_{\tau \alpha} (c_{\tau \alpha}^{-1}y) \cdots
\end{equation}
for $y \in k$. If $\alpha \in \Delta(T,B)$, then $c_\alpha = c_{\tau \alpha} = \cdots = 1$, so we have $U_{\bar{\alpha}} \cong {\mathbb G}_a$ and $x^{H_1}_{\bar{\alpha}}$ is given by the formula
\begin{equation} \label{type1_rootgroup}
x^{H_{1}}_{\bar{\alpha}}(y) = x_\alpha(y)\, x_{\tau \alpha}(y) \cdots .
\end{equation}

Next suppose $\alpha \in \Psi^+$ represents a set $S_\alpha$ of Type 2: $S_\alpha = \{ \alpha, \tau(\alpha), \beta \}$, where $\beta := \alpha + \tau \alpha$ is a root of $H^\circ$.  Then $\beta/2 = \bar{\alpha}$. Following \cite[Theorem 8.2]{St}, we may normalize the homomorphism $x_\beta$ such that $[x_\alpha(y), x_{\tau \alpha}(y')] = x_\beta(yy')$, where $[a,b] := a^{-1}b^{-1}ab$. We stress that $x_\beta$ depends on the choice of ordering $(\alpha, \tau \alpha)$ of the set $\{\alpha, \tau \alpha \}$. It is proved in \cite[Thm.~8.2(2)]{St} that there are nontrivial elements in $U^\tau$ of the form $x_{\alpha}(y_\alpha) x_{\tau \alpha}(y_{\tau \alpha}) x_{\beta}(y_\beta)$ only if $c_\alpha c_{\tau \alpha} = \pm 1$.  In fact we will always have $c_\alpha c_{\tau \alpha} = 1$: since we are dealing with root subgroups we may assume $G$ is adjoint and simple of type $A_{2n}$, and that $\tau$ is the unique (order two) element in ${\rm Aut}({\rm PGL}_{2n+1})$ which fixes the standard splitting and induces the order two diagram automorphism; then for all $y \in k$, $x_\alpha(y) = \tau^2x_{\tau^2\alpha}(y) = \tau x_{\tau \alpha}(c_\alpha y ) = x_{\alpha}(c_{\tau \alpha} c_\alpha y)$.  

Assume ${\rm char}(k) \neq 2$. As $c_{\alpha} c_{\tau \alpha} =1$ is automatic, according to the proof of \cite[Thm.8.2(2)]{St} we may define $x^{H_{1}}_{\bar{\alpha}} : \mathbb G_a \rightarrow H_1$ by
\begin{equation} \label{proto_type2_rtgrp_char_not_2}
x^{H_1}_{\bar{\alpha}}(y) = x_\alpha(y)\, x_{\tau \alpha}(c_\alpha y) \, x_{\beta}(-c_\alpha y^2/2).
\end{equation} 
Let $U_{\bar{\alpha}} \cong {\mathbb G}_a$ be the image of $x^{H_1}_{\bar{\alpha}}$.
If $\alpha \in \Delta(T,B)$ then $c_\alpha = c_{\tau \alpha}  =1$, and $x^{H_1}_{\bar{\alpha}}$ is given by
\begin{equation} \label{type2_rootgroup_char_not_2}
x^{H_{1}}_{\bar{\alpha}}(y) = x_\alpha(y) \, x_{\tau \alpha}(y) \, x_{\beta}(-y^2/2).
\end{equation}

Assume ${\rm char}(k) = 2$. Then following \cite[Thm.~8.2(2)]{St}, for $\alpha \in \Psi^+$ representing $S_\alpha$, $y_\alpha$ and $y_{\tau \alpha}$ are forced to be trivial, and $y_\beta$ ranges freely, so that we may define $x^{H_1}_{2\bar{\alpha}} : \mathbb G_a \rightarrow H_1$ by
\begin{equation} \label{type2_rootgroup_char=2}
x^{H_1}_{2\bar{\alpha}}(y) = x_\beta(y).
\end{equation}
For $i \in \{ 1, 2\}$, in all cases define $U_{i\bar{\alpha}}$ to be the image of $x^{H_1}_{i\bar{\alpha}}$ when $x^{H_1}_{i\bar{\alpha}}$ can be defined; otherwise set $U_{i\bar{\alpha}} = \{1\}$. Then \cite[Theorem 8.2]{St}~shows that $U^\tau$ is the product of the subgroups $U_{\bar{\alpha}}$ (or sometimes~$U_{2\bar{\alpha}}$) corresponding to the various $S_\alpha$'s.  In particular $U^\tau$ is connected. Further, $U_{\bar{\alpha}}$ (resp. $U_{2\bar{\alpha}}$) is isomorphic to ${\mathbb G}_a$ whenever $\alpha \in \Delta(T,B)$. This holds for more general $\alpha \in \Psi^+$ too, except possibly when $S_\alpha$ has type 1: {\em a priori} $U_{\bar{\alpha}}$ could be trivial if $\alpha \notin \Delta(T,B)$ (but see Remark \ref{nontriv_rem}). Thus property (iii) holds for $I = \langle \tau \rangle$. 

Now we consider (i). The argument of the {\em Lifting step} above used only Lemma \ref{root_lem} and property (iii), and so can be used here to show that $N^\tau \cap H^{\tau,\circ} \twoheadrightarrow W^\tau$. Let $R \subset H^{\tau,\circ}$ be the unipotent radical. By \cite[Cor.~7.4]{St}, $R$ is contained in a $\tau$-stable Borel subgroup of $H^\circ$, which we may assume to be $B$; hence $R \subset U^\tau$. But by the surjectivity of $N^\tau \cap H^{\tau,\circ} \rightarrow W^\tau$ and by (iii, iv), only the trivial subgroup of $U^\tau$ can be normalized by $N^\tau \cap H^{\tau,\circ}$. Hence $R =1$ and $H^{\tau,\circ}$ is reductive.  

Since $U^\tau$ is connected it follows that $B^{\tau,\circ} = T^{\tau,\circ} \cdot U^\tau$.  Also, $H^{\tau,\circ}/ B^{\tau,\circ}$ is proper, so $B^{\tau,\circ}$ is a parabolic subgroup of $H^{\tau,\circ}$. Thus $B^{\tau,\circ}$ is a Borel subgroup of $H^{\tau,\circ}$, being a connected solvable parabolic subgroup of a reductive group. It follows that $T^{\tau,\circ}$ is a maximal torus of $H^{\tau,\circ}$.

Finally, we need to construct the splitting $X^\tau$. If ${\rm char}(k) \neq 2$, then the definition of $x^{H_1}_{\bar{\alpha}}$ above shows that the simple roots for $\Psi(H^{\tau,\circ}, T^{\tau,\circ})$ are the averages $\bar{\alpha}$ of the $\tau$-orbits of the $\alpha \in \Delta(T, B)$. For a simple root $\alpha' \in \Delta(T^{\tau,\circ}, B^{\tau,\circ})$, let $$X_{\alpha'} := \sum_{\underset{\bar{\alpha} = \alpha'}{\alpha \in \Delta(T,B)}} X_\alpha.$$ 
One can check by taking differentials of (\ref{type1_rootgroup}) and (\ref{type2_rootgroup_char_not_2}) that $X_{\alpha'} \in {\rm Lie}(H^{\tau,\circ})_{\alpha'}$, and so $X = \sum_{\alpha'} X_{\alpha'}$ gives the desired splitting. If ${\rm char}(k) = 2$, we have to be more careful: ${\rm Lie}(H^{\tau, \circ})$ can be smaller than ${\rm Lie}(H^\circ)^\tau$, and in fact when $S_\alpha$ is Type 2, $X_\alpha + X_{\tau \alpha}$ will not belong to ${\rm Lie}(H^{\tau,\circ})$. Nevertheless, we can define $X^\tau$ to be the splitting corresponding to the collection of root-group homomorphisms
\begin{equation} \label{new_splitting}
\{ x^{H_{1}}_{\bar{\alpha}}(y)\} \bigcup \{ x^{H_{1}}_{2\bar{\alpha}}(y) \}
\end{equation}
where the first (resp.~second) collection in the union is indexed by the Type 1 (resp.~Type 2) subsets $S_\alpha$ (for $\alpha \in \Delta(T, B)$). 

It remains to prove (ii) when $I = \langle \tau \rangle$. First assume ${\rm char}(k) \neq 2$. Then $(\Psi^\tau)_{\rm red}$ is a reduced root system with simple roots $\Delta^\tau = \{ \bar{\alpha}  ~| ~ \alpha \in \Delta \}$.  But $\Psi(H^{\tau, \circ}, T^{\tau, \circ})$ is also a reduced root system and we saw in the proof of (i,iii) above that it also has $\Delta^\tau$ as its set of simple roots.  Hence $(\Psi^{\tau})_{\rm red} = \Psi(H^{\tau, \circ}, T^{\tau, \circ})$.  Next assume ${\rm char}(k) = 2$.  Now $(\Psi^\tau)^{\rm red}$ is a reduced root system with simple roots $\{ \bar{\alpha} \} \cup \{ 2\bar{\alpha}\}$ where the first (resp.~second) collection in the union is indexed by Type 1 (resp.~Type 2) subsets $S_\alpha$ (for $\alpha \in \Delta$). We saw above that this is precisely the set of simple roots for $\Psi(H^{\tau, \circ}, T^{\tau, \circ})$; hence $(\Psi^\tau)^{\rm red} = \Psi(H^{\tau, \circ}, T^{\tau, \circ})$.  In particular, as $\bar{\alpha} \in (\Psi^{\tau})^{\rm red}$ whenever $\alpha \in \Psi^+$ represents a Type 1 $S_\alpha$, we now see that $U_{\bar\alpha} \cong {\mathbb G}_a$ for such $\alpha$'s (cf.~Remark \ref{nontriv_rem}).

Thus we have proved (i-iii) hold when $I = \langle \tau \rangle$.  Note again that when $|\tau|$ is odd $\Psi^\tau$ is reduced.

Now suppose $I = S_3$, which will arise in the same way as in the proof of Lemma \ref{root_lem}.  Write $I = \langle \tau_1, \tau_2 \rangle$, where $\tau_1$ generates the normal subgroup of order $3$, and $\tau_2$ is of order $2$. Then by applying the above argument first with $\tau = \tau_1$ and then with $\tau = \tau_2$ (note that $\tau_2$ fixes the splitting $X^{\tau_1} = X$), we see that (i-iii) also hold in this case. 

Now consider the most general case, where $I$ is arbitrary. Let $Z$ denote the center of $H^\circ$. We have a short exact sequence 
\begin{equation} \label{adj_exact_seq1}
1 \rightarrow Z^I \rightarrow (H^\circ)^I \rightarrow (H^\circ_{\rm ad})^I \rightarrow H^1(I, Z).
\end{equation}
As $Z$ has finite $|I|$-torsion, we see $H^1(I,Z)$ is finite and thus $[(H^\circ)^I]^\circ$ surjects onto $[(H^\circ_{\rm ad})^I]^\circ$.  So we have an exact sequence for $H$
\begin{equation} \label{adj_exact_seq2}
1 \rightarrow Z^I \cap H^{I,\circ} \rightarrow H^{I,\circ} \rightarrow (H^\circ_{\rm ad})^{I, \circ} \rightarrow 1.
\end{equation}
This exact sequence shows that $H^{I, \circ}$ is reductive if $(H^\circ_{\rm ad})^{I,\circ}$ is reductive. Similarly, (i-iii) for $H^\circ$ follow formally from (i-iii) for $H_{\rm ad}^\circ$.

Thus, we are reduced to assuming $H = H^\circ_{\rm ad}$.  Then $H$ is a product of simple groups, which are permuted by $I$ and which each carry an action by the stabilizer subgroup of $I$.  We may therefore assume $H$ is simple, and the classification shows we may assume $I = \mathbb Z/2\mathbb Z$, $\mathbb Z/3\mathbb Z$, or $S_3$.  Each of these cases was handled above, and we conclude that (i-iii) indeed hold for adjoint groups.  
\end{proof}

\begin{Remark}
Some of Proposition \ref{Steinberg_style_prop} appears in \cite[Lemma A.1]{Ri}, but with the unnecessary assumption that the order $|I|$ is prime to ${\rm char}(k)$.  The latter assumption is indeed necessary to prove that $H^I$ is reductive, when $I$ is finite but is not assumed to fix any Borel pair $(T,B)$ (see \cite[Thm.~2.1, Rem.~3.5]{PY}, on which \cite{Ri} relies).  
\end{Remark}

\begin{lemma} \label{Z->>pi_0}

Assume $H, B, T, X,$ and $I$ are as in Proposition \ref{Steinberg_style_prop}. Let $Z$ denote the center of $H^\circ$. 
\begin{enumerate}
\item[(i)] Suppose $(H^\circ_{\rm ad})^I$ is connected. Then the natural map $Z^I \rightarrow \pi_0 \, (H^\circ)^I$ is surjective.
\item[(ii)] Let $T_{\rm ad}$ denote the image of $T$ in $(H^\circ)_{\rm ad}$. If $(T_{\rm ad})^I$ is connected, then $Z^I (H^\circ)^{I,\circ} = (H^\circ)^I$.
\end{enumerate}
\end{lemma}

\begin{proof}
For part (i), the key point is that $H^{I, \circ}$ surjects onto $(H_{\rm ad}^\circ)^I$ by (\ref{adj_exact_seq2}). Part (ii) follows immediately from (i) and Proposition \ref{Steinberg_style_prop}(d).
\end{proof}

\begin{lemma} \label{tau_conj_lem}
Suppose $H, B, T$ are as in Proposition \ref{Steinberg_style_prop}, and assume $I = \langle \tau \rangle$ fixes $B,T$ but not necessarily $X$. Also assume ${\rm char}(k)= 0$ and that $H$ is connected. In the group $H \rtimes \langle \tau \rangle$, the set  $(H \rtimes \tau)_{\rm ss}$ of semisimple elements in the coset $H \rtimes \tau$ is the set of $H$-conjugates of $T \rtimes \tau$.
\end{lemma}

\begin{proof}
Conjugates of elements in $T \rtimes \tau$, are semi-simple since $\tau$ has finite order and ${\rm char}(k) = 0$.  Conversely, suppose $h\tau$ is semi-simple. Since $H$ is connected, \cite[7.3]{St} ensures that some $\tau$-conjugate of $h$ lies in $B$; hence we may assume $h \in B$. Now by \cite[Theorem 7.5]{St} applied to the (disconnected) group $B \rtimes \langle \tau \rangle$, ${\rm Int}(h\tau)$ fixes a maximal torus $T' \subset B$.  Write $T' = bT b^{-1}$ for some $b \in B$.  We obtain $b^{-1}h \tau(b) \in N \cap B =T$, thus $h\tau$ is $H$-conjugate to $T \rtimes \tau$.
\end{proof}

\begin{Remark} For applications in the rest of this article, we only need the case $k = \mathbb C$ of these results, and so strictly for the present purposes the above exposition could have been shortened somewhat. However, eventually one hopes to develop a geometric Satake isomorphism for general groups {\em and} for coefficients in arbitrary fields $k$ (or even in arbitrary commutative rings; see \cite{MV}) and in those more general situations the above results will be needed.
\end{Remark}

We close this section with a lemma that will be needed for the proof of Theorem \ref{explicit_S(G)} to be given in $\S$\ref{pf_exp_S(G)_sec}. We keep the notation and hypotheses of Proposition \ref{Steinberg_style_prop}, except we assume ${\rm char}(k) = 0$ and $I = \langle \tau \rangle$. Fix $\lambda \in k^\times$ which is {\em not a root of unity}. As in the introduction, let $\mathcal P^\tau_{\rm reg}(H)$ denote the set of pairs $(h \rtimes \tau, e) \in [H \rtimes \tau]_{\rm ss} \times \mathcal N(H)$ with (i) ${\rm Ad}(h \rtimes \tau)(e) = \lambda e$, (ii) $\tau(e) = e$, and (iii) $e$ is a {\em principal} nilpotent element in ${\rm Lie}(H^\circ)$.  Note that $H$ acts on $\mathcal P^\tau_{\rm reg}(H)$ by the formula $g_1\cdot (g\rtimes \tau, e) = (g_1g\tau(g_1)^{-1} \rtimes \tau, {\rm Ad}(g_1)(e))$.  

Suppose $(g \rtimes \tau,e) \in \mathcal P^\tau_{\rm reg}(H)$. Associated to $e \in {\rm Lie}(H)$ is a uniquely determined Borel subgroup $B_e \subset H$, defined as follows (see \cite[$\S5.7$]{Ca}).  By the Jacobson-Morozov theorem, $e$ belongs to an ${\mathfrak sl}_2$-triple $\{ e, f, h \}$. This gives a copy of ${\mathfrak sl}_2$ inside ${\rm Lie}(H)$, well-defined up to $C_{H^\circ}(e)$-conjugacy. We can lift the Lie-algebra embedding ${\mathfrak sl}_2 \hookrightarrow {\rm Lie}(H)$ to a group embedding ${\rm SL}_2 \hookrightarrow H^\circ$, also well-defined up to $C_{H^\circ}(e)$-conjugacy. Consider the cocharacter $\gamma$ given by $\lambda \mapsto \begin{bmatrix} \lambda & 0 \\ 0 & \lambda^{-1} \end{bmatrix}$ (viewed inside $H^\circ$). Let $T$ be a maximal torus of $H^\circ$ containing $\gamma(k^\times)$. Define 
$$
P_e = \langle T, U_\alpha; \langle \alpha, \gamma \rangle \geq 0 \rangle.
$$
Then \cite[Prop.~5.7.1]{Ca} shows that $P_e$ is a well-defined parabolic subgroup of $H^\circ$ determined by $e$, and that $C_{H^\circ}(e) \subseteq P_e$.\footnote{Strictly speaking, \cite[Prop.~5.7.1]{Ca} pertains only to the adjoint group $(H^\circ)_{\rm ad}$, but the same proof applies.} In our case $e$ and hence the semisimple element $h$ is regular, meaning $\gamma$ is also regular. This implies that $P_e$ is a Borel subgroup, which we henceforth denote by $B_e$. Note that $B_e$ is preserved by $\tau$. Let $U_e$ denote the unipotent radical of $B_e$.

\begin{lemma} \label{nilpotent_lem} Assume ${\rm char}(k) = 0$, $I = \langle \tau \rangle$, and $\lambda \in k^\times$ is not a root of unity. Assume further that $H = Z(H) H^\circ$. Let $(g \rtimes \tau,e) \in \mathcal P^\tau_{\rm reg}(H)$, and suppose $T \subset B_e$ is a maximal torus containing $\gamma(k^\times)$ as above, so that we can write $e = \sum_{\alpha \in \Delta_e} X_\alpha$, where $\Delta_e := \Delta(T, B_e)$. Let $\delta \in T$ be an element such that ${\rm Ad}(\delta)(e) = \lambda e$. Then the $U_e$-orbit of $(g \rtimes \tau,e)$ contains an element of the form $(\delta z \rtimes \tau, e)$, where $z \in Z(H)$.
\end{lemma} 

\begin{proof}
First note that ${\rm Ad}(g \rtimes \tau)(e) = \lambda e$ implies ${\rm Ad}(g)(e) = \lambda e$, since $\tau$ fixes $e$.  Since $H = Z(H) H^\circ$, we may assume $g \in H^\circ$.  We have $\delta^{-1}g \in C_{H^\circ}(e) \subseteq B_e$.  Hence $g \in B_e$.  As in the proof of Lemma \ref{tau_conj_lem}, we may find $u \in U_e$ and $t \in T$ such that $g = u^{-1}t \tau(u)$. 

Recall $e = \sum_{\alpha \in \Delta_e} X_\alpha$ where $X_\alpha \neq 0,\, \forall \alpha$. Let $\beta$ range over $B_e$-positive roots with ${\rm ht}(\beta) \geq 2$, and for each choose $X_\beta \in {\rm Lie}(H^\circ)_\beta \backslash \{0 \}$.  
 We can write
\begin{align} 
{\rm Ad}(u)(e) &= e + \sum_{{\rm ht}(\beta) \geq 2} a_\beta X_\beta  \label{ad(u)_eq}\\
{\rm Ad}(\tau (u))(e) &= e + \sum_{{\rm ht}(\beta) \geq 2} a_\beta \, \tau(X_\beta)  \label{ad(tau(u))_eq} 
\end{align}
for certain scalars $a_\beta \in k$.  

Applying ${\rm Ad}(t)$ to (\ref{ad(tau(u))_eq}), we get
$$
{\rm Ad}(t \tau(u))(e) = {\rm Ad}(t)(e) + \sum_{{\rm ht}(\beta) \geq 2} a_{\beta}\, (\tau^{-1}\beta)(t)\, \tau(X_\beta).
$$
Applying ${\rm Ad}(u^{-1})$ to this and using the hypothesis on $g$ yields
$$
\lambda e = {\rm Ad}(u^{-1}t \tau(u))(e) = {\rm Ad}(t)(e) + \sum_{{\rm ht}(\beta) \geq 2} Y_\beta
$$
for certain $Y_\beta \in {\rm Lie}(H^\circ)_\beta$. So $Y_\beta = 0, \forall \beta$ and 
\begin{equation} \label{ad(t)_eq}
{\rm Ad}(t)(e) = \lambda e.
\end{equation}
As $\delta^{-1}t \in T$ and fixes $e$ under the adjoint action, we obtain $t = \delta \cdot z$ for some $z \in Z(H^\circ)$.

Our hypothesis on $g$ means that applying ${\rm Ad}(t)$ to (\ref{ad(tau(u))_eq}) yields $\lambda$ times (\ref{ad(u)_eq}).  But using (\ref{ad(t)_eq}), we see that ${\rm Ad}(t)(X_\alpha) = \lambda X_\alpha, \, \forall \alpha \in \Delta_e$, and thus we get
$$
\lambda e + \sum_{{\rm ht}(\beta) \geq 2} \lambda^{{\rm ht}(\beta)} a_{\beta} \tau(X_{\beta}) =
\lambda e + \sum_{{\rm ht}(\beta) \geq 2} \lambda a_{\beta}  X_\beta.
$$
Here we used that $\tau$ preserves heights. In fact, if we fix a height $h \geq 2$, and let $[\tau]$ denote the matrix given by $\tau$ and the basis $\{ X_\beta\}_{{\rm ht}(\beta) = h}$, we get an equality of column vectors of the form
$$
[\tau] [\lambda^h a_\beta]_\beta = [\lambda a_\beta]_\beta.
$$
Choose $N \geq 1$ with $[\tau]^N = {\rm id}$. This entails $a_\beta \lambda^{N(h-1)} = a_\beta$.  Since $\lambda$ is not a root of unity, we deduce that $a_\beta = 0$ for all $\beta$.  

Thus ${\rm Ad}(u)(e) = e$, and $u \cdot (g \rtimes \tau,e) = (\delta z \rtimes \tau, e)$, as desired.
\end{proof}

\section{The parameter space} \label{param_space_sec5}

Assume $G$ is a {\em quasi-split} connected reductive group over $F$, and fix an $F$-rational maximal torus/Borel subgroup $T \subset B$ thereof. Let $\widehat{G}$ be the complex dual group of $G$. By definition, it carries an action by the absolute Galois group $\Gamma$ over $F$, which factors through a finite quotient and fixes a splitting of the form $(\widehat{B},\widehat{T},\widehat{X})$ (cf. \cite[1.5]{cusptemp}), where we may assume $\widehat{T}$ is the complex dual torus for $T$. Note that since $G$ is quasi-split with $\Gamma$-fixed pair $(B,T)$, the $\Gamma$-action on $\widehat{T}$ inherited from $\widehat{G}$ agrees with that derived from the $\Gamma$-action on $X_*(T) = X^*(\widehat{T})$.\footnote{See the even more pedantic discussion of dual groups in the second paragraph of section \ref{trans_sec}.} We shall use this remark below, applied to the torus $T_{\rm sc}$.

\smallskip

The group $\widehat{G}^{I_F}$ is reductive by Proposition \ref{Steinberg_style_prop} and carries an action by $\tau = \Phi$ which fixes the splitting $(\widehat{T}^{I_F, \circ}, \widehat{B}^{I_F, \circ}, \widehat{X})$ (see Proposition \ref{Steinberg_style_prop}(a)).  Write $I$ for $I_F$ in what follows.

By Lemma \ref{tau_conj_lem} applied with $H = \widehat{G}^{I,\circ}$, we have a surjection 
$$
\widehat{T}^{I,\circ}  \rightarrow [\widehat{G}^{I,\circ} \rtimes \Phi]_{\rm ss}/ \widehat{G}^{I,\circ}.
$$
Let $\widehat{Z} = Z(\widehat{G})$. Since the torus $T_{\rm sc}$ in $G_{\rm sc}$ is $I$-induced, its dual torus $\widehat{T}_{\rm ad}$ is also $I$-induced, and so $(\widehat{T}_{\rm ad})^I$ is connected. Then from Lemma \ref{Z->>pi_0} (ii) with $H = \widehat{G}$, we see $\widehat{Z}^I \cdot \widehat{G}^{I,\circ} = \widehat{G}^I$.

Now multiplying on the left by $\widehat{Z}^I$, the above surjection gives rise to a surjection
\begin{equation} \label{1st_eq}
(\widehat{T}^{I})_{\Phi} \rightarrow [\widehat{G}^I \rtimes \Phi]_{\rm ss}/ \widehat{G}^I.
\end{equation}
Let $\widehat{N} = N(\widehat{G}, \widehat{T})$ and $\widehat{W} = W(\widehat{G}, \widehat{T})$. By Proposition \ref{Steinberg_style_prop}(c), $\widehat{W}^{I} = W(\widehat{G}^{I,\circ},\widehat{T}^{I,\circ})$, and $(\widehat{W}^{I})^{\Phi} = W(\widehat{G}^{\Gamma ,\circ}, \widehat{T}^{\Gamma, \circ})$.  Further, we see that $(\widehat{N}^I)^\Phi = \widehat{N}^\Gamma$ surjects onto $(\widehat{W}^I)^\Phi = \widehat{W}^\Gamma$.  We thus have a well-defined surjective map
\begin{equation} \label{2nd_eq}
(\widehat{T}^I)_\Phi/ \widehat{W}^\Gamma \rightarrow [\widehat{G}^I \rtimes \Phi]_{\rm ss}/ \widehat{G}^I.
\end{equation}

\begin{prop} \label{2nd_eq_bijective} The map \textup{(}\ref{2nd_eq}\textup{)} is bijective.
\end{prop}

\begin{proof}
It remains to prove the injectivity, which is similar to \cite[Prop.~11]{Mis}.  Suppose there exist $s, t \in \widehat{T}^I$ and $z g_0 \in Z^I \cdot \widehat{G}^{I,\circ}$ with $(zg_0)^{-1}s \Phi(zg_0) = t$.  Write $U$ for the unipotent radical of $\widehat{B}$.  Via the Bruhat decompsition write $g_0 = u_0  n_0 v_0$ where $n_0 \in \widehat{N}^I \cap \widehat{G}^{I,\circ}$, $v_0 \in U^I$, and $u_0 \in U^I \cap \, ^{n_0}\overline{U}^I$. We have
$$
s\Phi(u_0) \Phi(zn_0) \Phi(v_0) = u_0 (zn_0) v_0 t,
$$
and thus
$$
(s \Phi(u_0) s^{-1}) \cdot s\Phi(zn_0) \cdot \Phi(v_0) = u_0 \cdot (zn_0t) \cdot (t^{-1}v_0t).
$$
Uniqueness of the decomposition yields
$$
s\Phi(zn_0) = zn_0t.
$$
The image of $zn_0 \in \widehat{N}^I$ in $\widehat{W}^I$ is therefore $\Phi$-fixed, so lifts (cf. Prop. \ref{Steinberg_style_prop}) to some element $n_1 \in \widehat{N}^\Gamma$; write $zn_0 = t_1 n_1$ for some $t_1 \in \widehat{T}^I$. The resulting equation
$$
n_1^{-1}(t_1^{-1} s \Phi(t_1)) n_1 = t
$$
shows that $s$ and $t$ have the same image in $(\widehat{T}^I)_\Phi/ \widehat{W}^\Gamma$.
\end{proof}

\begin{cor} \label{alg_var_str_cor} If $G/F$ is quasi-split, the set $[\widehat{G}^{I_F} \rtimes \Phi]_{\rm ss}/\widehat{G}^{I_F}$ has the structure of an affine algebraic variety canonically isomorphic to $(\widehat{T}^{I_F})_\Phi/\widehat{W}^\Gamma$.
\end{cor}

\begin{Remark} \label{dual_innertwist_rem}
If $G$ is not quasi-split over $F$, then as in $\S\ref{trans_sec}$, we consider it as an inner form $(G, \Psi)$ of a group $G^*$ which is quasi-split over $F$.  Then $\Psi$ induces a canonical $\Gamma$-isomorphism of based root systems $\psi: \Psi_0(G) ~ \widetilde{\rightarrow} ~ \Psi_0(G^*)$. Following \cite[$\S1$]{cusptemp}, recall that a dual group for $G^*$ is a pair $(\widehat{G^*}, \iota)$, where $\widehat{G^*}$ is a connected reductive group over $\mathbb C$, where $\iota: \Psi_0(G^*)^\vee ~ \widetilde{\rightarrow} ~ \Psi_0(\widehat{G^*})$ is an $\Gamma^*$-isomorphism of based root systems, and where $\Gamma^*$ fixes some splitting for $\widehat{G^*}$.\footnote{We write $\Gamma^*$, $I^*_F$, etc., to indicate Galois actions on $G^*$.} If $(\widehat{G^*}, \iota)$ is a dual group for $G^*$, then $(\widehat{G^*}, \iota \circ \check{\psi}^{-1})$ is a dual group for $G$. Thus $(G^*, \Psi)$ gives rise to canonical identifications $^LG^* = \, ^LG$ and 
\begin{equation} \label{paramspace_twist}
[\widehat{G^*}^{I^*_F} \rtimes \Phi^*]_{\rm ss}/\widehat{G^*}^{I^*_F}  = [\widehat{G}^{I_F} \rtimes \Phi]_{\rm ss}/\widehat{G}^{I_F}.
\end{equation}
Thus the right hand side inherits the structure of an affine algebraic variety from the left hand side.
\end{Remark}

\section{Construction of parameters: quasi-split case} \label{param_space_sec6}

Now again assume $G/F$ is quasi-split. Let $A$ be a maximal $F$-split torus in $G$, and suppose $T = {\rm Cent}_G(A)$; let $W = W(G, T)$ and recall that since $G$ is quasi-split, $W^\Gamma$ is the relative Weyl group $W(G, A)$.  There is a $\Gamma$-equivariant isomorphism $W \cong \widehat{W}$. Putting this together with Proposition \ref{2nd_eq_bijective}  yields the following result.

\begin{prop} \label{3rd_eq_prop} Assume $G$ is quasi-split over $F$. There is a natural bijection
\begin{equation} \label{3rd_eq}
(\widehat{T}^{I_F})_{\Phi}/ W(G,A) ~ \widetilde{\rightarrow} ~ [\widehat{G}^{I_F} \rtimes \Phi]_{\rm ss}/ \widehat{G}^{I_F}.
\end{equation}
\end{prop}

Let $J \subset G(F)$ be any parahoric subgroup, and let $\pi \in \Pi(G/F, J)$. By \cite[$\S11.5$]{stable}, there exists a weakly unramified character $\chi \in (\widehat{T}^{I_F})_\Phi/W(G,A)$ such that $\pi$ is an irreducible subquotient of the normalized induction $i^G_B(\chi)$. 
\begin{defn} \label{qs_key_defn}
Define $s(\pi) \in S(G) := [\widehat{G}^{I_F} \rtimes \Phi]_{\rm ss}/\widehat{G}^{I_F}$ to be the image of $ \chi \in (\widehat{T}^{I_F})_\Phi/W(G,A)$ under the bijection (\ref{3rd_eq}).\footnote{For the independence of the map $\pi \mapsto s(\pi)$ from auxiliary choices such as $A$, see the discussion of (\ref{key_defn}).} 
\end{defn} 

We have the Bernstein isomorphism of \cite[11.10.1]{stable}\footnote{We use the letter $S$ for this map, because when $J = K$ it is just the Satake isomorphism (\ref{Sat_isom}).}
\begin{equation} \label{Bern_isom}
S: \mathcal Z(G(F),J) ~ \widetilde{\rightarrow} ~ \mathbb C[(\widehat{T}^{I_F})_\Phi/{W(G,A)}].
\end{equation}
By \cite[$\S11.8$]{stable}, $z \in \mathcal Z(G(F), J)$ acts on $i^G_B(\chi)^J$ by the scalar $S(z)(\chi)$. 
Then we have the following characterization of $s(\pi)$: {\em any element $z \in \mathcal Z(G(F), J)$ acts on $\pi^J$ by the scalar $S(z)(s(\pi))$}.

\begin{lemma} \label{comp_lem} The map $\pi \mapsto s(\pi)$ is compatible with change of level $J' \subset J$.
\end{lemma}

\begin{proof} Clearly $\Pi(G/F, J) \subset \Pi(G/F, J')$, and the compatibility simply reduces to the compatibility between Bernstein isomorphisms when $J' \subset J$. The latter follows from the construction of \cite[11.10.1]{stable}.
\end{proof}

\section{Second parametrization of $K$-spherical representations: quasi-split case} \label{2nd_param_sec}

Continue to assume $G$ is quasi-split over $F$, but take $J = K$ to be a maximal special parahoric subgroup.  In this case the Satake parameter can be described in another way.

\begin{theorem} \label{2nd_param_thm} Assume $G, K$ as above. We have the following parametrization of $\Pi(G/F,K)$
$$
\xymatrixcolsep{5pc}\xymatrix{
\Pi(G/F, K) \ar@{->}[r]^{(\ref{1st_param})}_{\sim} & (\widehat{T}^{I_F})_\Phi/W(G,A) \ar[r]^{(\ref{3rd_eq})}_{\sim} & [\widehat{G}^{I_F} \rtimes \Phi]_{\rm ss}/ \widehat{G}^{I_F}.
}
$$
\end{theorem}

We can also realize the Satake parameter $s(\pi)$ for $\pi \in \Pi(G/F,K)$ to be the  image of $\pi$ under this map. This image $s(\pi)$ may be characterized as follows: it is the unique element of the affine variety $[\widehat{G}^{I_F} \rtimes \Phi]_{\rm ss}/\widehat{G}^{I_F}$ such that
$$
{\rm tr}(f | \pi) = S(f)(s(\pi))
$$
where $S(f)$ is the Satake transform for any $f \in \mathcal H(G(F),K)$. (The Satake isomorphism (\ref{Sat_isom}) is just a specific instance of a Bernstein isomorphism and Lemma \ref{comp_lem} shows the two ways of constructing $s(\pi)$ coincide.)



\section{Review of transfer homomorphisms} \label{trans_sec}

In order to define Satake parameters for general groups, we need to recall the normalized transfer homomorphisms introduced in \cite[$\S11$]{stable}. Let $G^*$ be a quasi-split group over $F$. Let $F^s$ denote a separable closure of $F$, and set $\Gamma = {\rm Gal}(F^s/F)$. Recall that an inner form of $G^*$ is a pair $(G,\Psi)$ consisting of a connected reductive $F$-group $G$ and a $\Gamma$-stable $G^*_{\rm ad}(F^s)$-orbit $\Psi$ of $F^s$-isomorphisms $\psi: G \rightarrow G^*$.  The set of isomorphism classes of pairs $(G,\Psi)$ corresponds bijectively to $H^1(F, G^*_{\rm ad})$.  As before, we will write $\Gamma^*$, $I^*_F$, etc., to indicate Galois actions on $G^*$.

In the construction of transfer homomorphisms, we start with the choice of some primary data: $A$, $A^*$, and $\widehat{B^*} \supset \widehat{T^*}$. Here, $A$ (resp.~$A^*$) is a maximal $F$-split torus in $G$ (resp.~$G^*$). We will set $M = {\rm Cent}_G(A)$ and $T^* = {\rm Cent}_{G^*}(A^*)$, a maximal torus in $G^*$.  The Borel/torus pair $\widehat{B^*} \supset \widehat{T^*}$ in $\widehat{G^*}$ is specified as follows: we require $\widehat{T^*}, \widehat{B^*}$ to be part of some $\Gamma^*$-fixed splitting $(\widehat{T^*}, \widehat{B^*}, \widehat{X^*})$ (see Remark \ref{dual_innertwist_rem}).  Let $\iota$ be as in Remark \ref{dual_innertwist_rem}. Since $G^*$ is quasi-split, $\iota$ induces a $\Gamma^*$-isomorphism $X_*(T^*) ~ \widetilde{\rightarrow} ~ X^*(\widehat{T^*})$.

Now we make some secondary choices: choose an $F$-parabolic subgroup $P \subset G$ having $M$ as Levi factor, and an $F$-rational Borel subgroup $B^* \subset G^*$ having $T^*$ as Levi factor. Then there exists a unique parabolic subgroup $P^* \subset G^*$ such that $P^* \supseteq B^*$ and $P^*$ is $G^*(F^s)$-conjugate to $\psi(P)$ for every $\psi \in \Psi$.  Let $M^*$ be the unique Levi factor of $P^*$ containing $T^*$. Then define
$$
\Psi_M = \{ \psi \in \Psi ~ | ~ \psi(P) = P^*, \,\, \psi(M) = M^* \}.
$$
(Note we suppress the dependence of $\Psi_M$ on $P, B^*$.) The set $\Psi_M$ is a nonempty $\Gamma$-stable $M^*_{\rm ad}(F^s)$-orbit of $F^s$-isomorphisms $M \rightarrow M^*$, and so $(M,\Psi_M)$ is an inner form of $M^*$. Choose any $\psi_0 \in \Psi_M$.  Then since $\psi_0|A$ is $F$-rational, $\psi_0(A)$ is an $F$-split torus in $Z(M^*)$ and hence $\psi_0(A) \subseteq A^*$. 

The $F$-Levi subgroup $M^*$ corresponds to a $\Gamma^*$-invariant subset $\Delta_{M^*}$ of the $B^*$-positive simple roots $\widehat{\Delta} \subset X^*(T^*)$. It follows that $\iota(\Delta_{M^*}^\vee)$ is a set $\Delta_{\widehat{M^*}}$ of $\widehat{B^*}$-positive simple roots in $X^*(\widehat{T^*})$ for some uniquely determined $\Gamma^*$-stable Levi subgroup $\widehat{M^*} \supset \widehat{T^*}$. Note that $\iota$ defines a $\Gamma^*$-isomorphism $\iota: \Psi_0(M^*)^\vee ~ \widetilde{\rightarrow} ~  \Psi_0(\widehat{M^*})$. Writing $\widehat{X^*} = \{ \widehat{X^*}_\alpha \}_{\alpha \in \widehat{\Delta}}$, we see that $\widehat{M^*}$ has a $\Gamma^*$-fixed splitting $(\widehat{T^*}, \widehat{B^*}_{\widehat{M^*}}, \{ \widehat{X^*}_\alpha \}_{\alpha \in \Delta_{\widehat{M^*}}})$, where $\widehat{B^*}_{\widehat{M^*}} := \widehat{B^*} \cap \widehat{M^*}$.  Hence $(\widehat{M^*}, \iota)$ is a dual group for $M^*$. 

Thus, for every $\psi_0 \in \Psi_M$, we have a $\Gamma$-equivariant homomorphism $$\hat{\psi}_0: Z(\widehat{M}) ~ \widetilde{\rightarrow} ~ Z(\widehat{M^*}) \hookrightarrow \widehat{T^*}.$$ (See Remark \ref{dual_innertwist_rem}.)

We obtain a morphism of affine algebraic varieties 
\begin{align*}
t^{\widehat{T^*}, \widehat{B^*}}_{A^*,A}: (Z(\widehat{M})^{I_F})_{\Phi_F}/W(G,A) &\longrightarrow (\widehat{T^*}^{I^*_F})_{\Phi^*_F} /W(G^*, A^*) \\
\hat{m} &\longmapsto \hat{\psi}_0(\hat{m}).
\end{align*}
The morphism $t^{\widehat{T^*}, \widehat{B^*}}_{A^*, A}$ is independent of the choices of $P$ and $B^*$. Henceforth we will follow the notation of \cite[$\S12.2$]{HRo} and \cite[$\S11$]{stable}, by writing $t_{A^*, A}$ instead of $t^{\widehat{T^*}, \widehat{B^*}}_{A^*,A}$. 

We now recall the definition of a normalized version of $t_{A^*,A}$, for which we need to refine the choice of $\psi_0 \in \Psi_M$ somewhat. Following \cite[Lemma 11.12.4]{stable}, given the choice of $P \supset M$ and $B^* \supset T^*$ used to define $\Psi_M$, choose any {\em $F^{\rm un}$-rational} $\psi_0 \in \Psi_M$ and define a morphism of affine algebraic varieties
\begin{align} \label{t-tilde}
\tilde{t}_{A^*, A} : (Z(\widehat{M})^{I_F})_{\Phi_F}/W(G,A) &\longrightarrow (\widehat{T^*}^{I^*_F})_{\Phi^*_F} /W(G^*, A^*) \\
\hat{m} &\longmapsto \delta_{B^*}^{-1/2} \cdot \hat{\psi}_0(\delta_P^{1/2} \hat{m}). \notag
\end{align}
This makes sense as $\delta_P$ (resp.~$\delta_{B^*}$) is a weakly unramified character of $M(F)$ (resp.~$T^*(F)$), and so can be regarded as an element of $(Z(\widehat{M})^{I_F})_{\Phi_F}$ (resp.~$(\widehat{T^*}^{I^*_F})_{\Phi^*_F}$), by \cite[(3.3.2)]{stable}. 

\begin{lemma} \label{norm_trans_lemma}
The morphism $\tilde{t}_{A^*,A}$ is well-defined and independent of the choice of $P$, $B^*$, and $F^{\rm un}$-rational $\psi_0 \in \Psi_M$ used in its construction. 
\end{lemma}

\begin{proof}
The independence statement and the compatibility with the Weyl group actions are proved in \cite[11.12.4]{stable}.
\end{proof}

\begin{lemma} \label{immersion_lem} The morphism \textup{(}\ref{t-tilde}\textup{)} is a closed immersion.
\end{lemma}

\begin{proof}
We first prove that the map $(Z(\widehat{M})^{I_F})_{\Phi_F} \rightarrow (\widehat{T^*}^{I^*_F})_{\Phi^*_F}$ given by $\hat{m} \mapsto \delta_{B^*}^{-1/2} \hat{\psi}_0(\delta_P^{1/2}\hat{m})$ is a closed immersion. For this it is clearly enough to show that the {\em unnormalized} map $\hat{m} \mapsto \hat{\psi}_0(\hat{m})$ is a closed immersion. But this follows from the surjectivity of the corresponding map
\begin{equation} \label{corr_map}
t_{A^*,A}: X^*(\widehat{T^*})^{\Phi^*_F}_{I^*_F} \rightarrow X^*(Z(\widehat{M}))^{\Phi_F}_{I_F},
\end{equation}
which was proved in \cite[Remark 11.12.2]{stable}. In fact this is done by interpreting (\ref{corr_map}), via the Kottwitz isomorphism, as the natural map
\begin{equation} \label{Kottwitz_reinterpret}
\xymatrixcolsep{2pc}\xymatrix{T^*(F)/T^*(F)_1 \ar[r] & M^*(F)/M^*(F)_1 \ar[r]^{\psi_0^{-1}}_{\sim} & M(F)/M(F)_1.}
\end{equation}
We therefore have a surjective normalized variant
\begin{equation} \label{norm_Kottwitz_reinterpret}
\xymatrixcolsep{2pc}\xymatrix{\tilde{t}_{A^*, A}: T^*(F)/T^*(F)_1 \ar[r] & M^*(F)/M^*(F)_1 \ar[r]^{\psi_0^{-1}}_{\sim} & M(F)/M(F)_1.}
\end{equation}

Now recall that \cite[11.12.3]{stable} constructs a bijective map
\begin{equation} \label{bij_map}
\xymatrixcolsep{4pc}\xymatrix{
W(G,A) \ar[r]^{\psi_0^\natural \,\,\,\,\,\,\,\,\,\,\,\,\,\,\,\,}_{\sim \,\,\,\,\,\,\,\,\,\,\,\,\,\,\,\,\,\,} & W(G^*,A^*)/W(M^*, A^*)}
\end{equation}
defined as follows. Let $F^{\rm un}$ be the maximal unramified extension of $F$ in $F^s$, and let $L$ denote the completion of $F^{\rm un}$. Let $S^*$ be the $F^{\rm un}$-split component of $T^*$. Choose a maximal $F^{\rm un}$-split torus $S \subset G$ which is defined over $F$ and which contains $A$, and set $T = {\rm Cent}_G(S)$. Choose $\psi_0 \in \Psi_M$ such that $\psi_0$ is defined over $F^{\rm un}$ and has $\psi_0(S) = S^*$, hence also $\psi_0(T) = T^*$. Now suppose $w \in W(G,A)$. We may choose a representative $n \in N_G(S)(L)^{\Phi_F}$ (cf.~\cite{HRo}). There exists $m^*_n \in N_{M^*}(S^*)(L)$ such that $\psi_0(n)m^*_n \in N_{G^*}(A^*)(F)$. Then define $\psi_0^\natural(w)$ to be the image of $\psi_0(n)m^*_n$ in $W(G^*,A^*)/W(M^*,A^*)$.

Now the desired surjectivity of 
\begin{equation} \label{want_surj}
\xymatrixcolsep{5pc}\xymatrix{
\mathbb C[T^*(F)/T^*(F)_1]^{W(G^*,A^*)} \ar[r]^{\tilde{t}_{A^*A}} &  \mathbb C[M(F)/M(F)_1]^{W(G,A)}}
\end{equation}
follows without difficulty using the surjectivity of (\ref{norm_Kottwitz_reinterpret}) and the isomorphism (\ref{bij_map}), because $W(M^*,A^*)$ and $N_{M^*}(S^*)(L)$ act trivially on $M^*(F)/M^*(F)_1$. Indeed, for $m \in M(F)/M(F)_1$, define $\Sigma_m \in \mathbb C[M(F)/M(F)_1]^{W(G,A)}$ by $\Sigma_m := \sum_{w \in W(G,A)} w \cdot m$.  Suppose $t^* \in T^*(F)/T^*(F)_1$ maps to $m$ under (\ref{norm_Kottwitz_reinterpret}), and define $\Sigma_{t^*} \in \mathbb C[T^*(F)/T^*(F)_1]^{W(G^*, A^*)}$ by $\Sigma_{t^*} := \sum_{w^* \in W(G^*, A^*)} w^* \cdot t^*$. Then (\ref{want_surj}) sends $\Sigma_{t^*}$ to $|W(M^*, A^*)| \cdot \Sigma_m$.
\end{proof}


Recall the definition of the normalized transfer homomorphism on the level of Bernstein centers.
\begin{defn} \label{tilde-t_defn} (\cite[11.12.1]{stable})
Let $J \subset G(F)$ and $J^* \subset G^*(F)$ be any parahoric subgroups and choose maximal $F$-split tori $A$ resp.~$A^*$ to be in good position\footnote{This means that in the Bruhat-Tits building $\mathcal B(G_{\rm ad}, F)$, the facet corresponding to $J$ is contained in the apartment corresponding to $A$.} relative to $J$ resp.~$J^*$. Then we define the {\em normalized transfer homomorphism} $\tilde{t}: \mathcal Z(G^*(F), J^*) \rightarrow \mathcal Z(G(F), J)$ to be the unique homomorphism making the following diagram commute
$$
\xymatrix{
\mathcal Z(G^*(F), J^*) \ar[r]^{\tilde{t}} \ar[d]_{\wr}^{S} & \mathcal Z(G(F), J) \ar[d]_{\wr}^{S}  \\
 \mathbb C[X^*(\widehat{T^*})^{\Phi_F^*}_{I^*_F}]^{W(G^*,A^*)} \ar[r]^{\tilde{t}_{A^*,A}} & \mathbb C[X^*(Z(\widehat{M}))^{\Phi_F}_{I_F}]^{W(G,A)}.
}$$
\end{defn}
We use $S$ to denote the Bernstein isomorphisms described in \cite[11.10.1]{stable}. As explained in \cite[Def.~11.12.5]{stable}, $\tilde{t}$ is independent of the choices for $A, A^*$, and $\widehat{B^*} \supset \widehat{T^*}$, and is a completely canonical homomorphism.

\begin{cor} \textup{(}of Lemma \ref{immersion_lem} \textup{)} \label{tilde-t_surj}
The normalized transfer homomorphism $\tilde{t}: \mathcal Z(G^*(F), J^*) \rightarrow \mathcal Z(G(F), J)$ is {\em surjective}.
\end{cor}

We now present an alternative way to characterize the maps $\tilde{t}$, reformulating slightly \cite[11.12.6]{stable}.

\begin{prop} Choose $A, A^*, \psi_0 \in \Psi_M$ as needed in Definition \ref{tilde-t_defn}. For each subtorus $A_L \subseteq A$, let $L = {\rm Cent}_G(A_L)$ and $L^* = \psi_0(L)$, so that $\psi_0$ restricts to an inner twising $L \rightarrow L^*$ of $F$-Levi subgroups of $G$ resp.~$G^*$. Set $J_L = J \cap L(F)$. Then the family of normalized transfer homomorphisms $\tilde{t} : \mathcal Z(L^*(F), J^*_{L^*}) \rightarrow \mathcal Z(L(F), J_L)$ is the unique family with the following properties:
\begin{enumerate}
\item[(a)] The $\tilde{t}$ are compatible with the constant term homomorphisms $c^G_L$, in the sense that the following diagrams commute for all $L$:
$$
\xymatrix{
\mathcal Z(G^*(F), J^*) \ar@{^{(}->}[d]_{c^{G^*}_{L^*}} \ar[r]^{\tilde t} & \mathcal Z(G(F), J) \ar@{^{(}->}[d]_{c^G_L} \\
\mathcal Z(L^*(F), J^*_{L^*}) \ar[r]^{\tilde{t}} & \mathcal Z(L(F), J_L).}
$$ 
\item[(b)] For $L = M$ and $z \in \mathcal Z(M^*(F), J^*_{M^*})$, the function $\tilde{t}(z)$ is given by integrating $z$ over the fibers of the Kottwitz homomorphism $\kappa_{M^*}(F)$. \textup{(}Note $M_{\rm ad}$ is anisotropic over $F$.\textup{)}
\end{enumerate}
\end{prop}
The constant term homomorphisms here are defined in \cite[11.11]{stable} as follows: suppose $Q = LR$ is an $F$-rational parabolic subgroup with Levi factor $L$ and unipotent radical $R$. Given $z \in \mathcal Z(G(F),J)$, define $c^G_L(z) \in \mathcal Z(L(F), J_L)$ by
$$
c^G_L(z)(l) = \delta_Q^{1/2}(l) \int_{R(F)} z(lr) \, dr = \delta_Q^{-1/2}(l) \int_{R(F)} z(rl) \, dr,
$$
for $l \in L(F)$, where ${\rm vol}_{dr}(J \cap R(F)) = 1$.  It is proved as in \cite[Lemma 4.7.2]{H09} that $c^G_L(z)$ really does belong to the center of $\mathcal H(L(F), J_L)$ and is independent of the choice of $Q$ having $L$ as a Levi factor.

\section{Construction of parameters: general case} \label{param_const_gen}

Suppose $G$ is any connected reductive group over $F$, and $J \subset G(F)$ is a parahoric subgroup.  Fix our primary data $A, A^*$ and $\widehat{B^*} \supset \widehat{T^*}$ as in the construction of $\tilde{t}_{A^*,A}$ in (\ref{t-tilde}).

Let $\pi \in \Pi(G/F, J)$. By \cite[$\S11.5$]{stable}, there exists a weakly unramified character $\chi \in (Z(\widehat{M})^{I_F})_\Phi/W(G,A)$ such that $\pi$ is an irreducible subquotient of the normalized induction $i^G_P(\chi)$. 

\begin{defn} Define $s(\pi) \in [\widehat{G}^{I_F} \rtimes \Phi]_{\rm ss}/\widehat{G}^{I_F}$ to be the image of $ \chi \in (Z(\widehat{M})^{I_F})_\Phi/W(G,A)$ under the map
\begin{equation} \label{key_defn}
\xymatrixcolsep{1.5pc}\xymatrix{
(Z(\widehat{M})^{I_F})_\Phi/W(G,A) \ar@{^{(}->}[r]^{\,\,\, \tilde{t}_{A^*, A}} & (\widehat{T^*}^{I^*_F})_\Phi/W(G^*, A^*) \ar[r]^{(\ref{3rd_eq})}_{\sim} & [\widehat{G^*}^{I^*_F} \rtimes \Phi^*]_{\rm ss}/\widehat{G^*}^{I^*_F} \ar@{=}[r]^{\,\,\, (\ref{paramspace_twist})} & [\widehat{G}^{I_F} \rtimes \Phi]_{\rm ss}/\widehat{G}^{I_F}.}
\end{equation} 
Define $S(G)$ to be the image of this map, which is a closed subvariety of $[\widehat{G}^{I_F} \rtimes \Phi]_{\rm ss}/\widehat{G}^{I_F}$ by Lemma \ref{immersion_lem}. 
\end{defn}

Let us prove that the set $S(G)$ and the element $s(\pi) \in S(G)$ are independent of the primary choices $A, A^*$, $\widehat{B^*} \supset \widehat{T^*}$ we have made in their construction.  Because $\widehat{G^*}^{\Gamma^*}$ acts transitively on $\Gamma^*$-fixed splittings (\cite[1.7]{cusptemp}), the map (\ref{key_defn}) is already independent of the pair $\widehat{B^*} \supset \widehat{T^*}$.  The independence of the map $\pi \mapsto s(\pi)$ from $A$ 
(resp.~$A^*$) results from the fact that any other choice for $A$ (resp.~$A^*$) would be $G(F)$-(resp.~$G^*(F)$-)conjugate to it.

Now suppose $A$, $J$ are as above. Then we have the Bernstein isomorphism of \cite[11.9.1]{stable}
\begin{equation*} 
S: \mathcal Z(G(F),J) ~ \widetilde{\rightarrow} ~ \mathbb C[(Z(\widehat{M})^{I_F})_\Phi/{W(G,A)}].
\end{equation*}
By \cite[$\S11.8$]{stable}, $z \in \mathcal Z(G(F), J)$ acts on $i^G_P(\chi)^J$ by the scalar $S(z)(\chi)$. 
Therefore we have the following characterization of $s(\pi)$: choose any parahoric 
subgroup $J^* \subset G^*(F)$; then for all $z^* \in \mathcal Z(G^*(F), J^*)$ we have
\begin{equation*}
{\rm tr}(\tilde{t}(z^*) \, | \, \pi) = {\rm dim}(\pi^J) \, S(z^*)(s(\pi)).
\end{equation*}

In particular, when $G$ is quasi-split, the map $\pi \mapsto s(\pi)$ defined here coincides with the map defined in Definition \ref{qs_key_defn}. Further, in the general case $\pi \mapsto s(\pi)$ is compatible with change of level $J' \subset J$ in the same sense as in Lemma \ref{comp_lem}.

\section{Second parametrization of $K$-spherical representations: general case} \label{2nd_param_gen_sec}

Let $K \subset G(F)$ be a special maximal parahoric subgroup. Putting together the isomorphism (\ref{1st_param}) with the map (\ref{key_defn}), we obtain the following.

\begin{theorem} \label{gen_2nd_param_thm} There is a canonical parametrization of $\Pi(G/F,K)$
$$
\xymatrixcolsep{4pc}\xymatrix{
\Pi(G/F, K) \ar@{->}[r]^{\,\,\,\,\, \pi \mapsto s(\pi)}_{\,\,\,\,\, \sim} & S(G) \ar@{^{(}->}[r] & [\widehat{G}^{I_F} \rtimes \Phi]_{\rm ss}/ \widehat{G}^{I_F}.
}
$$
Furthermore, $S(G) = [\widehat{G}^{I_F} \rtimes \Phi]_{\rm ss}/ \widehat{G}^{I_F}$ if and only if $G/F$ is quasi-split. 
\end{theorem}

\begin{proof}
The parametrization is immediate, and the ``only if'' results from the strict inequality ${\rm dim}\, (Z(\widehat{M})^{I_F})_\Phi = {\rm dim}\, (Z(\widehat{M^*})^{I^*_F})_{\Phi^*} < {\rm dim} \, (\widehat{T^*}^{I^*_F})_{\Phi^*}$ if $M$ is not a maximal torus in $G$. (The inequality follows from Lemma \ref{dim_lem} below.) This proves items (A) and (B) of Theorem \ref{main_thm}.
\end{proof}

\begin{lemma} \label{dim_lem} For any connected reductive $F$-group $G$, the dimension of the diagonalizable group $(Z(\widehat{G})^{I_F})_\Phi$ is the rank of the maximal $F$-split torus in the center of $G$.
\end{lemma}

\begin{proof}
Fix an $F$-rational maximal torus $T \subset G$ and set $T_{\rm der} = T \cap G_{\rm der}$. Define the cocenter torus $D = G/G_{\rm der} = T/T_{\rm der}$. The torus $Z(G)^\circ$ is isogenous to $D$, hence there is a perfect $\Gamma$-equivariant $\mathbb Q$-valued pairing between  $X^*(Z(G)^\circ)_\mathbb Q$ and $X_*(D)_\mathbb Q$. A result of Borovoi gives a $\Gamma$-equivariant isomorphism $X_*(T)/X_*(T_{\rm sc}) =  X^*(Z(\widehat{G}))$, where $T_{\rm sc}$ denotes the pull-back of $T_{\rm der}$ along the covering $G_{\rm sc} \rightarrow G_{\rm der}$. It follows that $X_*(D)_\mathbb Q = X^*(Z(\widehat{G}))_\mathbb Q$. We obtain a perfect $\mathbb Q$-valued pairing between $\big(X^*(Z(G)^\circ)\big)_\Gamma \otimes \mathbb Q$ and $\big(X^*(Z(\widehat{G}))\big)^\Gamma \otimes \mathbb Q = X^*\big((Z(\widehat{G})^{I_F})_\Phi\big) \otimes \mathbb Q$.  The lemma follows.
\end{proof}

\section{Proof of Theorem \ref{explicit_S(G)}} \label{pf_exp_S(G)_sec}

Thanks to (\ref{delta_identity}) below, the equivalence $(i) \Leftrightarrow (iii)$ is fairly obvious from the definition of $S(G)$. We concentrate on $(i) \Leftrightarrow (ii)$. We retain the notation from $\S \ref{trans_sec}$.  

\begin{lemma} \label{Ad_lem} Suppose that the principal nilpotent element $\widehat{X^*}_{\widehat{M^*}} : = \sum_{\alpha \in \Delta_{\widehat{M^*}}} \widehat{X^*}_\alpha$ is part of a $\Gamma^*$-fixed splitting $(\widehat{T^*}, \widehat{B^*}_{\widehat{M^*}}, \widehat{X^*}_{\widehat{M^*}})$ for $\widehat{M^*}$.  Let $\widehat{m^*} \in (Z(\widehat{M^*})^{I^*_F})_{\Phi^*}$.  Then 
\begin{equation} \label{Ad_eq}
{\rm Ad}(\delta^{-1/2}_{B^*_{M^*}} \widehat{m^*} \rtimes \Phi^*)(\widehat{X^*}_{\widehat{M^*}}) = q_F \widehat{X^*}_{\widehat{M^*}}.
\end{equation}
\end{lemma}

\begin{proof}
Thanks to the Kottwitz isomorphism
$${\rm Hom}_{\rm grp}(T^*(F)/T^*(F)_1, \mathbb C^\times) = (\widehat{T^*}^{I^*_F})_{\Phi^*},$$ 
$\delta^{-1/2}_{B^*_{M^*}}$ is naturally an element of $(\widehat{T^*}^{I^*_F})_{\Phi^*}$. The left hand side of (\ref{Ad_eq}) is well-defined. Let $\varpi \in F^\times$ be any uniformizer, and let $\rho^*_{\widehat{M^*}}$ be the half-sum of the $\widehat{B^*}_{\widehat{M^*}}$-positive roots in $X^*(\widehat{T^*})$. 
The homomorphism $X^*(\widehat{T^*}) \rightarrow \mathbb C^\times$ given by $\lambda \mapsto \delta_{B^*_{M^*}}^{-1/2}(\varpi^\lambda)$ corresponds to an element of $\widehat{T^*}^{I^*_F}$ which projects to $\delta_{B^*_{M^*}}^{-1/2}$ under $\widehat{T^*}^{I^*_F} \rightarrow (\widehat{T^*}^{I^*_F})_{\Phi^*}$; denote this element also by $\delta^{-1/2}_{B^*_{M^*}}$. It is clear that 
${\rm Ad}(\delta^{-1/2}_{B^*_{M^*}})$ acts on $\widehat{X^*}_{\widehat{M^*}}$ by the scalar $\delta^{-1/2}_{B^*_{M^*}}(\varpi^\alpha) = | \varpi^{-\langle \alpha, \rho^*_{\widehat{M^*}} \rangle}|_F = q_F$ (here $\alpha \in \Delta_{\widehat{M^*}}$ is arbitrary). The result follows because ${\rm Ad}(\widehat{m^*} \rtimes \Phi^*)$ fixes $\widehat{X^*}_{\widehat{M^*}}$.
\end{proof}

As in $\S \ref{trans_sec}$, we fix $\psi_0 \in \Psi_M$ as needed to define the normalized transfer homomorphism $\tilde{t}_{A^*, A}$. We have the identity
\begin{equation} \label{delta_identity}
\delta_{B^*_{M^*}}^{-1/2} = \delta_{B^*}^{-1/2} \hat{\psi}_0(\delta_P^{1/2})
\end{equation}
in $\widehat{T^*}^{I^*_F}$. Recall we have a canonical identification $[\widehat{G^*}^{I^*_F} \rtimes \Phi^*]_{\rm ss}/ \widehat{G^*}^{I^*_F} = [\widehat{G}^{I_F} \rtimes \Phi]_{\rm ss}/ \widehat{G}^{I_F}$. Therefore elements of $S(G)$ can be represented by elements of the form $\delta_{B^*_{M^*}}^{-1/2}\widehat{m^*}$ where $\widehat{m^*}$ is as in Lemma \ref{Ad_lem}. Further, by Proposition \ref{Steinberg_style_prop}(a), 
$\widehat{M^*}^{I^*_F}$ has a splitting of the form $(\widehat{T^*}^{I^*_F}, \widehat{B^*}^{I^*_F}_{\widehat{M^*}}, \widehat{X^*}_{\widehat{M^*}})$ where $\widehat{X^*}_{\widehat{M^*}}$ is as in Lemma \ref{Ad_lem}.  That lemma therefore implies that every element of $S(G)$ has the property stated in Theorem \ref{explicit_S(G)} (ii).

Conversely, suppose $(\hat{g} \rtimes \Phi, x) \in \mathcal P^\Phi_{\rm reg}(\widehat{M}^{I_F})$. We want to prove that $\hat{g} \rtimes \Phi$ belongs to $S(G)$. As we may work entirely in $\widehat{M^*}^{I^*_F}$, we may as well assume $M^* = G^*$. Now $x \in \mathcal N(\widehat{G^*}^{I^*_F})$ has $\Phi^*(x) = x$, so in fact $x$ is a principal nilpotent in ${\rm Lie}(\widehat{G^*}^{\Gamma^*_F})$. By Proposition \ref{Steinberg_style_prop}(a),  there is a splitting $(\widehat{T^*}^{\Gamma^*_F}, \widehat{B^*}^{\Gamma^*_F}, \widehat{X^*})$ for $\widehat{G^*}^{\Gamma^*_F}$, where 
$\widehat{X^*} =\sum_{\alpha \in \Delta_{\widehat{G^*}}} \widehat{X^*}_\alpha$ comes from a $\Gamma^*_F$-fixed splitting $(\widehat{T^*}, \widehat{B^*}, \widehat{X^*})$ for $\widehat{G^*}$. Being a principal nilpotent element in $\mathcal N(\widehat{G^*}^{\Gamma^*_F})$, $x$ is $\widehat{G^*}^{\Gamma^*_F}$-conjugate to such an element $\widehat{X^*}$; hence we may assume $x = \widehat{X^*}$. We apply Lemma \ref{nilpotent_lem} with $H = \widehat{G^*}^{I^*_F}$, $\tau = \Phi^*$, $\lambda = q_F$, $e = \widehat{X^*}$, and $\delta = \delta^{-1/2}_{B^*}$, where $(T^*, B^*)$ corresponds to $(\widehat{T^*}, \widehat{B^*})$.  Lemma \ref{nilpotent_lem} asserts that we may assume $(\hat{g} \rtimes \Phi^*, \widehat{X^*}) = (\delta_{B^*}^{-1/2} \hat{z} \rtimes \Phi^*, \widehat{X^*})$ for some $\hat{z} \in Z(\widehat{G^*}^{I^*_F})$.  By Lemma \ref{Z->>pi_0}(ii),  $\widehat{G^*}^{I^*_F} = Z(\widehat{G^*})^{I^*_F} \cdot \widehat{G^*}^{I^*_F, \circ}$, so we may write $\hat{z} = \hat{z}_1 \cdot \hat{z}_2$ where $\hat{z}_1 \in Z(\widehat{G^*})^{I^*_F}$ and $\hat{z}_2 \in Z(\widehat{G^*}^{I^*_F, \circ})$. But $\hat{z}_2$ is an element of $\widehat{T^*}^{I^*_F, \circ}$ such that ${\rm Ad}(\hat{z}_2)(\widehat{X^*}) = \widehat{X^*}$. 
Thus in fact $\hat{z}_2$ belongs to $Z(\widehat{G^*})$, and hence $\hat{z}$ belongs to $Z(\widehat{G^*})^{I^*_F}$.  This proves that $(\widehat{g} \rtimes \Phi^*, x)$ belongs to $S(G)$.
\qed

\section{A transfer map $\Pi(G,K) \rightarrow \Pi(G^*,K^*)$} \label{transfer_map_sec}

Let $K \subset G(F)$ and $K^* \subset G^*(F)$ be special maximal parahoric subgroups.  We shall define an operation
\begin{align*}
\Pi(G/F, K) ~ &{\hookrightarrow} ~ \Pi(G^*/F, K^*) \\
\pi &\mapsto \pi^*
\end{align*}
which is dual to $\tilde{t}: \mathcal H(G^*(F), K^*) \rightarrow \mathcal H(G(F), K)$. We identify $^LG = \, ^LG^*$ as in Remark \ref{dual_innertwist_rem}. Given $\pi$ we have $s(\pi) \in S(G) \subseteq [\widehat{G^*}^{I^*_F} \rtimes \Phi^*]_{\rm ss}/ \widehat{G^*}^{I^*_F} = S(G^*)$.  
\begin{defn}
We define $\pi^*  \in \Pi(G^*/F, K^*)$ to be the unique isomorphism class with $s(\pi^*) = s(\pi)$.
\end{defn}

Clearly $\pi^*$ is characterized by the equalities for all $f^* \in \mathcal H(G^*(F), K^*)$
\begin{equation*} 
{\rm tr}(\tilde{t}(f^*) \, | \, \pi)  = S(\tilde{t}(f^*))(s(\pi)) = S(f^*)(s(\pi^*)) = {\rm tr}(f^* \, | \, \pi^*).
\end{equation*}
The middle equality follows from the diagram in Definition \ref{tilde-t_defn} (taking $J = K$ and $J^* = K^*$).
We remark that the character identity directly characterizes $\pi$ in terms of $\pi^*$ because $f^* \mapsto \tilde{t}(f^*)$ gives a {\em surjective} map $\mathcal H(G^*(F), K^*) \rightarrow \mathcal H(G(F), K)$ (Cor.~\ref{tilde-t_surj}).

\section{Relation with local Langlands correspondence} \label{rel_LLC_sec}

\subsection{Construction of $s$-parameter in Deligne-Langlands correspondence}
The Satake parameter $s(\pi)$ should give us part of the local Langlands parameter associated to $\pi \in \Pi(G/F, J)$.  

\begin{conj} \label{LLC_relation}
Let $W'_F := W_F \rtimes \mathbb C$ be the Weil-Deligne group.  If $\pi \in \Pi(G/F, J)$ has local Langlands parameter $\varphi_\pi: W'_F \rightarrow \, ^LG$, then 
\begin{equation} \label{s-param_eq}
\varphi_\pi(\Phi) = s(\pi)
\end{equation}
as elements in $[\widehat{G} \rtimes \Phi]_{\rm ss}/ \widehat{G}$.
\end{conj}
Note that we still denote by $s(\pi)$ its image under the natural map $[\widehat{G}^{I_F} \rtimes \Phi]_{\rm ss} / \widehat{G}^{I_F} \rightarrow [\widehat{G} \rtimes \Phi]_{\rm ss}/ \widehat{G}$.

\begin{Remark}
Put another way, the conjecture predicts the $s$-parameter in the Deligne-Langlands triple $(s,u, \rho)$ which is hypothetically attached to a representation $\pi$ with Iwahori-fixed vectors (note that the works of Kazhdan-Lusztig \cite{KL}, Lusztig \cite{L1,L2} construct the entire triple unconditionally for many $p$-adic groups, but not for the most general $p$-adic groups).
\end{Remark}

\begin{Remark}
This is similar to \cite[Thm.~2]{Mis}, which discusses the case where $G$ is quasi-split and split over a tamely ramified extension of $F$.
\end{Remark}

When $G/F$ is quasi-split, (\ref{s-param_eq}) is predicted by the compatibility of the local Langlands correspondence (LLC) with normalized parabolic induction, as follows.  Recall the property LLC+ (\cite[$\S5.2$]{stable}), which is LLC for $G$ and all of its $F$-Levi subgroups, plus the compatibility of infinitesimal characters (restrictions of Langlands parameters to $W_F$) with respect to normalized parabolic induction. Write $G,T, B$, etc.~in place of $G^*, T^*, B^*$ etc.~ from $\S\ref{trans_sec}$. By \cite[$\S11.5$]{stable}, there is a weakly unramified character $\chi$ on $T(F)$ such that $\pi$ is a subquotient of $i^G_B(\chi)$. Assuming LLC+ holds, we expect $\varphi_\pi|_{W_F} = \varphi_\chi|_{W_F}$, the latter taking values in $^LT \hookrightarrow \,^LG$. Now, the local Langlands correspondence for tori implies that $\varphi_\chi$ exists unconditionally, and has $\varphi_\chi(\Phi) = \chi \rtimes \Phi$, where on the right hand side $\chi$ is viewed as an element of $(\widehat{T}^{I_F})_{\Phi_F}$ via the Kottwitz isomorphism. But clearly $s(\pi) = \chi \rtimes \Phi$ as well.

Thus, the conjecture gives an essentially new prediction only when $G$ is not quasi-split.  In fact, its content is that the normalized transfer homomorphisms, used to define $s(\pi)$ in the general case, are really telling us what $\varphi_\pi(\Phi)$ should be. For example, if $D/F$ is a quaternion algebra over its center $F$ and $G = D^\times$, $J = \mathcal O^\times_D$, and $\pi = {\bf 1}_{D^\times}$ (the trivial representation of $D^\times$ on $\mathbb C$), then Conjecture \ref{LLC_relation} predicts that $\varphi_\pi(\Phi) = 
{\rm diag}(q^{-1/2}, q^{1/2}) \rtimes \Phi$, where $q$ is the cardinality of the residue field of $F$.  This is indeed the case (cf.~e.g.~Lemma \ref{JL_lemma} or \cite[Thm.~4.4]{PrRa}).

\subsection{Proof of Conjecture \ref{LLC_relation} for inner forms of ${\rm GL}_n$}

Suppose $G^* = {\rm GL}_n$ and that $G = {\rm GL}_m(D)$, where $D$ is a central division algebra over $F$ with ${\rm dim}_F(D) = d^2$, and $m$ is an integer with $n = md$. We will identify ${\rm GL}_m(D)$ with an inner form $(G, \Psi)$ of $G^*$.  We will assume that LLC+ holds for the group $G$. Of course the local Langlands correspondence is known for ${\rm GL}_n$, and it is also known that ${\rm GL}_n$ satisfies LLC+ (cf.~\cite[Rem.~13.1.1]{HRa} or \cite{Sch}). The local Langlands correspondence for the inner form $G$ is also well-understood, and presumably the property LLC+ similarly holds for $G$. This can likely be extracted from some recent works such as \cite{ABPS, Bad1, HiSa}.  We will not verify that $G$ satisfies LLC+ here, and instead we leave this task to another occasion.\footnote{The property LLC+ for $G$ has recently been verified by Jon Cohen and will appear as part of his forthcoming University of Maryland PhD thesis.}

Choose $A, A^*, \psi_0 \in \Psi_M$ as in (\ref{t-tilde}), and assume $A^*$ is the standard diagaonal torus in $G^* = {\rm GL}_n$. Given $\pi \in \Pi(G/F, J)$, its supercuspidal support is $(M, \chi)_G$ for some unramified character $\chi \in X(M)$.  The $F$-Levi subgroup $M \subset G$ (resp.~$M^* := \psi_0(M) \subset G^*$) has the form
$$
M \cong \prod_{i=1}^r {\rm GL}_{m_i}(D), \hspace{.2in}  (\mbox{resp.}~ M^* \cong \prod_{i=1}^r {\rm GL}_{n_i}, \,\, \mbox{a standard Levi subgroup of ${\rm GL}_n$})
$$
for some integers $m_i, n_i$ with $m_id= n_i, \, \forall i$ and $\sum_i n_i = n$.  It is harmless to assume that $\psi_0$ induces for all $i$ an inner twisting ${\rm GL}_{m_i}(D) \rightarrow {\rm GL}_{n_i}$ which is the identity on $F^s$-points (only the Galois actions differ). It is also harmless to assume that $\psi_0 \in \Psi_M$ where $\Psi_M$ is defined as in $\S\ref{trans_sec}$ using the standard upper triangular Borel subgroup $B^* \subset {\rm GL}_n$, and that $B^*_{M^*} := B^* \cap M^*$ has the form $\prod_i B^*_i$ where each $B^*_i$ is the upper triangular Borel subgroup of ${\rm GL}_{n_i}$. 

Write $\chi = \chi_1 \boxtimes \cdots \boxtimes \chi_r$ where $\chi_i \in X({\rm GL}_{m_i}(D))$. By LLC+ for $G$, we have equality of $\widehat{G}$-conjugacy classes
\begin{equation*}
\varphi_\pi(\Phi) =\varphi_{\chi_1}(\Phi) \times \cdots \times \varphi_{\chi_r}(\Phi).
\end{equation*}
Let ${\rm Nrd}_i$ denote the reduced norm homomorphism ${\rm GL}_{m_i}(D) \rightarrow \mathbb G_m(F)$. 
We may write $\chi_i = \eta_i \circ {\rm Nrd}_i$ for a unique unramified character $\eta_i: F^\times \rightarrow \mathbb C^\times$. Use the same symbol $\eta_i$ to 
denote $\eta_i(\varpi_F) \in \mathbb C^\times$ (here $\varpi_F \in F^\times$ is a 
uniformizer corresponding to $\Phi$ under the Artin reciprocity map).  The Langlands dual of the homomorphism ${\rm Nrd}_i$ is the diagonal embedding ${\rm diag}_i: \mathbb G_m(\mathbb C) \rightarrow {\rm GL}_{n_i}(\mathbb C)$. If $z_{\eta_i} \in Z^1(W_F, \mathbb G_m(\mathbb C))$ (resp.~$z_{\chi_i} \in Z^1(W_F, Z({\rm GL}_{n_i}(\mathbb C)))$) is a 1-cocycle corresponding to $\eta_i$ (resp.~$\chi_i$) under Langlands duality for tori (resp.~quasi-characters), we have $z_{\eta_i}(\Phi) = \eta_i \in \mathbb C^\times$ (resp. $z_{\chi_i}(\Phi) = {\rm diag}_i(\eta_i) \in Z({\rm GL}_{n_i}(\mathbb C))$).

The local Langlands correspondence for ${\rm GL}_{m_i}(D)$ respects twisting by unramified characters (cf.~e.g.~\cite[(4.0.5)]{stable}). We can view the representation $\chi_i$ as the twist of the trivial representation by the quasi-character $\chi_i$.  So in view of the above paragraph we have 
\begin{equation*}
\varphi_{\chi_i}(\Phi) = z_{\chi_i}(\Phi) \, \varphi_{{\bf 1}_i}(\Phi) = {\rm diag}_i(\eta_i) \, \varphi_{{\bf 1}_i}(\Phi),
\end{equation*} 
where ${\bf 1}_i$ is the trivial representation of ${\rm GL}_{m_i}(D)$ on $\mathbb C$.  Thus we have
\begin{equation*}
\varphi_\pi(\Phi) = {\rm diag}_1(\eta_1) \, \varphi_{{\bf 1}_1}(\Phi) \times \cdots \times {\rm diag}_r(\eta_r) \, \varphi_{{\bf 1}_r}(\Phi).
\end{equation*} 

\begin{lemma} \label{JL_lemma} In the notation above we have $\varphi_{{\bf 1}_i}(\Phi) = \delta^{-1/2}_{B^*_i} \rtimes \Phi$, where the modulus character is viewed as a diagonal element in ${\rm GL}_{n_i}(\mathbb C)$.
\end{lemma}

\begin{proof}
Let ${\bf St}_i$ (resp.~${\bf St}^*_i$) denote the Steinberg representation of ${\rm GL}_{m_i}(D)$ (resp.~${\rm GL}_{n_i}(F)$). Note that this has the same supercuspidal support as ${\bf 1}_i$ (resp.~${\bf 1}^*_i$). By LLC+ for ${\rm GL}_{m_i}(D)$, we see that
$
\varphi_{{\bf 1}_i}(\Phi) = \varphi_{{\bf St}_i}(\Phi).
$
The Jacquet-Langlands correspondence gives a distinguished bijection between the sets of isomorphism classes of essentially square-integrable smooth irreducible representations
\begin{equation*}
{\rm JL} \, : \, \Pi^2({\rm GL}_{m_i}(D)) ~ \widetilde{\rightarrow} ~ \Pi^2({\rm GL}_{n_i}(F)).
\end{equation*}
The Langlands parameter of $\pi_i \in \Pi^2({\rm GL}_{m_i}(D))$ is that of ${\rm JL}(\pi_i) \in \Pi^2({\rm GL}_{n_i}(F))$ (cf.~e.g.~\cite{HiSa} or \cite{Bad1}).  Furthermore, ${\rm JL}({\bf St}_i) = {\bf St}^*_i$ (cf.~\cite[$\S7.2$]{Bad2}).  Thus we get
\begin{equation*}
\varphi_{{\bf 1}_i}(\Phi) = \varphi_{{\bf St}_i}(\Phi)  = \varphi_{{\bf St}^*_i}(\Phi) = \delta_{B^*_i}^{-1/2} \rtimes \Phi,
\end{equation*}
the last equality because ${\bf St}^*_i$ is a quotient of $i^{{\rm GL}_{n_i}}_{B^*_i}(\delta_{B^*_i}^{-1/2})$.
 \end{proof}
Therefore we have
\begin{equation} \label{last_thing}
\varphi_\pi(\Phi) = [\prod_{i} {\rm diag}_i(\eta_i) \, \delta^{-1/2}_{B^*_i}] \rtimes \Phi.
\end{equation}
It is easy to see, using the equalities $\prod_i \delta^{-1/2}_{B^*_i} = \delta_{B^*_{M^*}}^{-1/2} =  \delta^{-1/2}_{B^*}\hat{\psi}_0(\delta_{P}^{1/2})$ and (\ref{t-tilde}), that (\ref{last_thing}) is the image of $s(\pi)$ in $[\widehat{G} \rtimes \Phi]_{\rm ss}/\widehat{G}$. This completes the proof of Theorem \ref{2nd_main_thm}. \qed

\subsection{Compatibility with generalized Jacquet-Langlands correspondence}
Now return to the usual notation, where $G$ is general and is identified with an inner form $(G,\Psi)$ of a quasi-split group $G^*$. 
Let us identify $^LG = \, ^LG^*$ as in Remark \ref{dual_innertwist_rem}.

Given $\pi \in \Pi(G/F, J)$, we may choose any $\pi^* \in \Pi(G^*/F, J^*)$ such that $s(\pi^*) = s(\pi)$.  Note that if $J^* = K^*$, then $\pi^*$ is unique, but in general it will not be.

Since $s(\pi) = s(\pi^*)$ by construction of $\pi^*$, we expect $\varphi_\pi(\Phi) = \varphi_{\pi^*}(\Phi)$.  Since $\pi$ and $\pi^*$ are $J$-(resp.~$J^*$)-spherical, $\varphi_\pi$ and $\varphi_{\pi^*}$ should satisfy $\varphi_\pi(I_F) = \varphi_{\pi^*}(I_F) = 1 \rtimes I_F$, and so we expect $\varphi_{\pi}|_{W_F} = \varphi_{\pi^*}|_{W_F}$.  This is compatible with what a ``generalized Jacquet-Langlands correspondence'' would entail, at least on the level of infinitesimal classes (cf.~\cite[$\S5.1$]{stable}). Namely, $\pi$ should give rise to the composition
$$
\xymatrix{
W'_F \ar[r]^{\varphi_\pi} & \, ^LG \ar@{=}[r] & \, ^LG^* }
$$
which we call $\varphi^*$, which in turn should give rise to an $L$-packet $\Pi_{\varphi^*}$ for the group $G^*$. The map $\pi \mapsto \Pi_{\varphi^*}$ would be part of a ``generalized Jacquet-Langlands correspondence''.  However, usually we would not expect $\pi^* \in \Pi_{\varphi^*}$.  For example, if $D/F$ and $G = D^\times$ are as above, $J = \mathcal O^\times_D$, $G^* = {\rm GL}_2$, $J^* = {\rm GL}_2(\mathcal O_F)$, and $\pi = {\bf 1}_{D^\times}$, then $\pi^* = {\bf 1}_{{\rm GL}_2(F)}$, while $\Pi_{\varphi^*} = {\rm JL}(\pi)$ is the Steinberg representation of ${\rm GL}_2(F)$.

On the other hand, if we restrict to $W_F$, we get an agreement of infinitesimal characters $\varphi^*|_{W_F} = \varphi_\pi|_{W_F} = \varphi_{\pi^*}|_{W_F}$. Thus, while $\pi^*$ might sometimes not belong to the $L$-packet $\Pi_{\varphi^*}$, it will always belong to the infinitesimal class $\Pi_{\varphi^*|_{W_F}}$ containing  $\Pi_{\varphi^*}$.

\bigskip

\noindent {\em Acknowledgements}. I thank R.~Kottwitz for a useful conversation about the scope of the results in $\S\ref{fixed-pt_sec}$. Further, I am grateful to Kottwitz and also to X.~He, M.~Solleveld, and X.~Zhu for making some helpful remarks on a preliminary version of this article. I also thank the referee for his/her suggestions and remarks.

\small
\bigskip
\obeylines
\noindent
University of Maryland
Department of Mathematics
College Park, MD 20742-4015 U.S.A.
email: tjh@math.umd.edu

\end{document}